\documentclass{amsart}
\usepackage{graphicx,psfrag}
\usepackage{amsmath,amssymb}
\usepackage{enumerate}
\usepackage{amsthm}
\newtheorem{theorem}{Theorem}
\theoremstyle{plain}
\newtheorem{claim}{Claim}
\newtheorem*{corollary}{Corollary}
\newtheorem*{proposition}{Proposition}
\newtheorem{lemma}{Lemma}
\newtheorem*{algorithm}{Algorithm}
\newtheorem*{keller}{Theorem (Keller (1984))}
\newtheorem*{theorem3}{Theorem 3}

\theoremstyle{definition}
\newtheorem*{definition}{Definition}

\theoremstyle{remark}
\newtheorem*{remark}{Remark}
\newtheorem*{remarks}{Remarks}

\newtheorem*{example}{Example}
\newtheorem*{acknowledgements}{Acknowledgements}
\newcommand{\R}{\mathbb{R}}
\newcommand{\Z}{\mathbb{Z}}

\newcommand{\BV}{\mathrm{BV}}
\renewcommand{\P}{\mathcal{P}}
\renewcommand{\S}{\mathcal{S}}
\newcommand{\B}{\mathcal{B}}

\newcommand{\J}{\mathcal{J}}
\newcommand{\T}{\mathcal{T}}
\newcommand{\esssup}{\mathop{\mathrm{ess\,sup}}}
\newcommand{\essinf}{\mathop{\mathrm{ess\,inf}}}
\newcommand{\var}{\mathop{\mathrm{var}}}

\newcommand{\spec}{\mathop{\mathrm{spec}}}
\newcommand{\pp}{\ a.e.\: }
\renewcommand{\mod}[1]{\ (\mathrm{mod}\ #1)}

\newcommand{\U}[2]{U_{#1}(#2)} % Eigenspaces of the limiting symmetrized operator
\newcommand{\V}[2]{V_{#1}(#2)} % Flag of nested subspaces
\newcommand{\W}[2]{W_{#1}(#2)} % Oseledets spaces
\newcommand{\Wpush}[3]{W_{#1}^{(#3)}(#2)} % Image of V under action of A^(n)

\begin{document}

\title[Coherent structures for Perron--Frobenius cocycles]
{Coherent structures and isolated spectrum \\
for Perron--Frobenius cocycles}

\date{9th April 2008}

\author{Gary Froyland}
\address{School of Mathematics and Statistics, 
University of New South Wales, Sydney, NSW 2052, AUSTRALIA}
\email{g.froyland@unsw.edu.au}

\author{Simon Lloyd}
\address{School of Mathematics and Statistics, 
University of New South Wales, Sydney, NSW 2052, AUSTRALIA}
\email{s.lloyd@unsw.edu.au}

\author{Anthony Quas}
\address{Department of Mathematics and Statistics,
University of Victoria, Victoria, BC V8W 3R4, CANADA}
\email{aquas@uvic.ca}

\begin{abstract}
We present an analysis of one-dimensional models of dynamical
systems that possess ``coherent structures'';  global structures
that disperse more slowly than local trajectory separation.  We
study cocycles generated by expanding interval maps and the rates of
decay for functions of bounded variation under the action of the
associated Perron--Frobenius cocycles.

We prove that when the generators are piecewise affine and share a
common Markov partition, the Lyapunov spectrum of the
Perron--Frobenius cocycle has at most finitely many isolated points.
Moreover, we develop a strengthened version of the Multiplicative
Ergodic Theorem for non-invertible matrices and construct an
invariant splitting into Oseledets subspaces.

We detail examples of cocycles of expanding maps with isolated
Lyapunov spectrum and calculate the Oseledets subspaces, which lead
to an identification of the underlying coherent structures.

Our constructions generalise the notions of almost-invariant and
almost-cyclic sets to non-autonomous dynamical systems and provide a
new ensemble-based formalism for coherent structures in
one-dimensional non-autonomous dynamics.
\end{abstract}

\maketitle
\section{Introduction}
Transport and mixing processes play an important role in many natural
phenomena and their mathematical analysis has received considerable
attention in the last two decades.  The \emph{geometric approach} to
transport includes the study of invariant manifolds, which may act as
barriers to particle transport and inhibit mixing.  So-called
\emph{Lagrangian coherent structures} were introduced (\cite{HY00,H01}) 
as finite-time proxies for invariant manifolds in non-autonomous
settings.  The \emph{ergodic-theoretic approach} to transport
includes the study of relaxation of initial ensemble densities to an
invariant density, with a special focus on initial densities that
relax more slowly than suggested by the rate of local trajectory
separation.  Such slowly decaying ensembles have been studied as
``strange eigenmodes'' (\cite{LH04,PP03,PPE07} in fluids and have 
been used to identify almost-invariant sets \cite{DJ99,F05,FP08,F08}). 
Until now, a suitable framework for the ergodic-theoretic approach 
that deals with truly non-autonomous dynamics has been lacking.  
The main aim of this work is to develop the fundamental structures 
and results that will support a non-autonomous theory for an 
ensemble-based approach to coherent structures.

We study non-autonomous one-dimensional dynamical systems that are given by
compositions of expanding interval maps, and their action on ensembles
represented by probability densities.  The time evolution of a density is
given by the Perron--Frobenius operator. For a single piecewise-expanding
map these densities evolve toward an equilibrium distribution which is
absolutely continuous (see \cite{LY73}); when the map is transitive, this
equilibrium distribution is also unique. Thus the equilibrium
distribution is an eigenfunction of the Perron--Frobenius operator with
eigenvalue $1$.  The exponential rate of convergence to equilibrium is
governed by the spectrum of the Perron--Frobenius operator. When
restricted to the space of functions of bounded variation (BV), the
Perron--Frobenius operator is quasicompact (see \cite{HK82}), meaning
that there are only finitely many spectral points of modulus greater than
the essential spectral radius, and each is an isolated eigenvalue of
finite multiplicity. We will say an eigenvalue is \emph{exceptional} if
it is different from $1$ and has modulus greater than the essential
spectral radius. It is known that in the BV setting the essential
spectral radius is determined by the long-term rate of separation of
nearby trajectories. Eigenfunctions corresponding to exceptional
eigenvalues relax more slowly to equilibrium than suggested by the local
separation of trajectories. The existence of such eigenfunctions has been
attributed to the presence of ``almost-invariant sets'' (see
\cite{DJ99,DFS00,F07}).

Exceptional eigenvalues have previously been found by considering
piecewise-affine expanding maps with a Markov partition (\cite{B96},
\cite{DFS00}, \cite{KR04}). When restricted to
the space of step-functions constant on the Markov partition
intervals, the associated Perron--Frobenius operator becomes a finite
dimensional operator.  In the present work we extend these results to
the non-autonomous setting. Instead of iterating a single map, we
consider a cocycle of maps and its associated Perron--Frobenius
cocycle. The appropriate way to describe exponential rate of
convergence to equilibrium is via the \emph{Lyapunov spectrum} of the
Perron--Frobenius cocycle.  As the Perron--Frobenius operator is a 
Markov operator, the Lyapunov spectrum is contained in the interval 
$[-\infty,0]$. We look for \emph{exceptional} Lyapunov
exponents, namely those greater than the essential upper bound of the
Lyapunov spectrum but less than zero.

We obtain a Lyapunov spectral decomposition for the Perron--Frobenius
cocycle into invariant subspaces with given Lyapunov exponents (see
Corollary \ref{cor:decomp}). This relies on a new version of the
Multiplicative Ergodic Theorem (see Theorem \ref{thm:main}), which
provides an invariant \underline{splitting} into Oseledets spaces
\emph{even when the generators are non-invertible}. Our new version
strengthens the standard Multiplicative Ergodic Theorem (see, for 
example \cite[Theorem 3.4.1]{A98}) where only an invariant 
\underline{flag} of nested subspaces is supplied. 

We demonstrate the existence of slow-mixing coherent
structures by constructing periodic (see Theorem \ref{thm:pereg}) 
and non-periodic (see Theorem \ref{thm:nonpereg}) examples of Lebesgue 
measure-preserving one-sided cocycles with exceptional Lyapunov 
exponents.  In each case, we calculate algebraically the Oseledets 
subspaces associated with the largest exceptional exponent and verify 
that the second largest Oseledets space captures the coherent structures. 

Finally, we present an algorithm for approximating the Oseledets 
splitting, which is based on a new computational approach suggested by the 
proof of Theorem 3. We demonstrate the effectiveness of the algorithm, 
by approximating some Oseledets subspaces numerically.

\section{Preliminaries}

We study the Perron--Frobenius operator of compositions of expanding
maps. We first introduce the necessary notation and relevant results
for autonomous systems, and then extend this to the non-autonomous
case.

\subsection{Autonomous systems}

We say that $T:I\to I$, where $I=[0,1]$ or $S^1$, is an
\emph{expanding map} if there exists a finite partition
$a_0=0<a_1<\ldots<a_M=1$ such that, for each $i=1,\ldots,M$, $T$ is
continuous on $(a_{i-1},a_i)$ and extends to a $C^2$ map on
$[a_{i-1},a_i]$ satisfying $|DT|_{(a_{i-1},a_i)}|>1$.

The \emph{Perron--Frobenius} operator for an expanding map $T:I\to I$
is defined, for an $L^1$ function $f:I\to\R$, by
\begin{align}
\P f(x) = \sum_{y\in T^{-1}(x)} \frac{f(y)}{|D T(y)|}.
\end{align}

In \cite{LY73}, the Perron--Frobenius operator is used to prove that
expanding maps have an absolutely continuous invariant probability
measure. The key step of their proof is to show that the
Perron--Frobenius operator contracts the norm on a suitable space of
functions: the functions of bounded variation.

The \emph{variation} of a function $f:I\to\mathbb{R}$ on a subinterval
$A\subset I$ is defined by
\begin{align*}
\var_A f:= \var_{x\in A} f(x)= \sup_{k\in\mathbb{N}}
\left\{\sum_{i=0}^k|f(x_i)-f(x_{i-1})|:x_i\in A,\ x_0<\cdots<x_k\right\}.
\end{align*}
Given $f\in L^\infty\subset L^1$, the variation is defined by
$\var_I f = \inf\{\var_I g: f=g \pp \}$. We denote by $\BV$ the Banach space
$$
\BV=\left\{f\in L^\infty: \var_I f<\infty\right\},
$$
equipped with the norm $\|f\|=\max\{\|f\|_{L^1},\var_I f \}$.  
We denote Lebesgue measure on $I$ by $m$, and $f\in\BV$ is called a
\emph{(probability) density} if $f\geq 0$ on $I$ (and $\|f\|_{L^1}=1$).
We write $\BV_+=\{f\in\BV:f\geq 0\}$.

The Perron--Frobenius operator is \emph{Markov}: that is, if $f\in L^1$
is a density, then $\P f$ is also a density and $\|\P
f\|_{L^1}=\|f\|_{L^1}$. A probability density $f^*$ satisfying $\P f^*=
f^*$ is an invariant probability density for $T$.

Keller \cite{K84} shows that the Perron--Frobenius operator of an
expanding map has at most finitely many exceptional eigenvalues.

\begin{keller}\label{thm:Keller}
  Given an expanding interval map $T:I\to I$, its Perron--Frobenius
  operator $\P$ acting on $\BV$ has essential spectral radius
$$
\theta:=\lim_{n\to\infty} \sup_{x\in I} 
\left( \frac{1}{|D(T^n)(x)|} \right)^{1/n},
$$
and all other spectral points are isolated eigenvalues of finite
multiplicity.
\end{keller}

Exceptional eigenvalues have a distinguished dynamical significance as
their eigenfunctions are associated with relaxation to equilibrium at
exponential rates slower than the rate suggested by the average local
separation of trajectories $\theta$. For example, if $\P g=\lambda g$
with $\theta<|\lambda|<1$ then an initial density $f^*+\alpha g$,
$\alpha\neq 0$ will relax to $f^*$ at a rate slower than $\theta$.

Dellnitz and Junge \cite{DJ99} suggested that positive real
Perron--Frobenius eigenvalues near to 1 correspond to
\emph{almost-invariant sets}; more precisely, they suggested the
sets $A^+:=\{g>0\}$ and $A^-:=\{g\le 0\}$ formed an almost-invariant
partition of the state space.  Dellnitz et al.~\cite{DFS00} showed the
converse, presenting a class of interval maps with almost-invariant
sets and proving the existence of exceptional eigenvalues. 
Froyland \cite{F07} constructed a two-dimensional hyperbolic map 
with almost-invariant sets and proved the existence of an exceptional 
eigenvalue. Numerical methods have been developed (\cite{DJ99,F05,F08}) 
for the computation of exceptional eigenfunctions and almost-invariant 
sets;  these have been applied successfully in molecular dynamics 
(\cite{SHD99}), astrodynamics (\cite{D+05}), and ocean 
circulation (\cite{F+07}).

Our intent in the present work is to generalise the notion of
almost-invariant sets in autonomous systems to that of coherent
structures in non-autonomous systems.  The latter will represent
structures that are perhaps quite mobile, but disperse at rates
slower than suggested by local trajectory separation.

\subsection{Non-autonomous systems}

We will examine exceptional spectral points in the non-autonomous
case, and study compositions of expanding maps taken from a finite
collection, and composed in order according to given sequences.

Let $\sigma$ be an ergodic automorphism of a probability space
$(\Omega,\mathcal{H},p)$ that preserves the probability $p$.  Given a
measurable/topological/vector space $X$, a \emph{(one-sided) cocycle}
over $\sigma$ is a function $H:\Z^+\times \Omega\times X\to X$ with
the properties that for all $x\in X$ and $\omega\in\Omega$:
\begin{itemize}
\item $H(0,\omega,x)=x$;
\item for all $m,n\in\Z^+$, $H(m+n,\omega,x)=H(m,\sigma^n\omega,H(n,\omega,x))$.
\end{itemize}
We sometimes write $H^{(n)}(\omega)(x)$ for $H(n,\omega,x)$, and
$H(\omega)(x)$ for $H(1,\omega,x)$.  The \emph{generator} of a cocycle
$H$ is the mapping $\tilde{H}:\Omega\to \mathrm{End}(X)$ given by
$\tilde{H}(\omega)=H(\omega)$.  Since the cocycle is uniquely
determined by $\tilde{H}$, we occasionally refer to $\tilde{H}$ itself
as the cocycle when no confusion can occur.

In the sequel, $\sigma$ will frequently be a (left) shift, defined by 
$(\sigma\omega)_i=\omega_{i+1}$, acting on a two-sided
sequence space $(\Omega,\mathcal{F},p)$ on $K$ symbols, where $\Omega\subset
(Z_K)^\Z$, $Z_K=\{1,\ldots,K\}$, is invariant under $\sigma$. The 
shift $\sigma$ preserves the probability $p$ and is ergodic with respect
to $p$.

\begin{definition}
  Let $\{T_i\}_{i\in Z_K}$, be a collection of expanding maps of $I$,
  and let $\P_i:\BV\to\BV$ be the Perron--Frobenius operator
  associated to $T_i:I\to I$.  The \emph{map cocycle generated by
    $\{T_i\}_{i\in Z_K}$}, denoted by $\Phi:\Z^+\times \Omega\times
  I\to I$, is defined to be the one-sided cocycle with generator
  $\tilde{\Phi}(\omega)=T_{\omega_0}\in \{T_i\}_{i\in Z_K}$.
  Associated to $\Phi$ is the \emph{Perron--Frobenius cocycle}
  $\P:\Z^+\times \Omega\times \BV\to\BV$, which is defined to be the
  one-sided cocycle with generator
  $\tilde{\P}(\omega)=\P_{\omega_0}\in\{\P_i\}_{i\in Z_K}$.
\end{definition}

Notice that even though we use a two-sided shift space, we only form
one-sided cocycles not two-sided cocycles. This is because the expanding
maps are non-invertible, as are their Perron--Frobenius operators.

We say a cocycle is \emph{periodic} if the underlying shift space
$\Omega$ is generated by a single element: that is, there exists
$R\in\mathbb{N}$, called the \emph{period}, and $\omega\in\Omega$ such
that $\Omega=\{\omega,\sigma\omega,\ldots\sigma^{R-1}\omega\}$; we say a
cocycle is \emph{autonomous} if $\Omega$ contains a single element.

\section{Quasicompactness of the transfer cocycle}

Information about the exponential decay rates of the Perron--Frobenius
cocycle is given by its Lyapunov spectrum.

\begin{definition}
  We denote by $\lambda(\omega,f)$ the \emph{Lyapunov exponent} of
  $f\in\BV$, defined
$$
\lambda(\omega,f)=\limsup_{n\to\infty}\frac{1}{n}\log \|\P^{(n)}(\omega)
f\|.
$$
We define the \emph{Lyapunov spectrum} $\Lambda(\P(\omega))\subset \R$
of the Perron--Frobenius cocycle at $\omega$ to be the set
$$
\Lambda(\P(\omega)):= \{\lambda(\omega,f):f\in\BV\}.
$$
\end{definition}

The exponential rate of decay that can be expected purely from the
local expansion is the \emph{essential upper bound}
$$
\vartheta(\omega):=-\inf_{x\in I} \limsup_{n\to\infty}\frac{1}{n}\log
|D\Phi^{(n)}(\omega)(x)|<0,
$$
which we denote by $\vartheta$ if independent of $\omega$.  Points in
$\Lambda(\P(\omega))$ that are greater than $\vartheta(\omega)$
indicate the presence of large-scale structures that reduce the rate
of mixing of the system, except for the maximal Lyapunov exponent, $0$, 
which is associated with an invariant density.
We refer to Lyapunov spectral points in the
interval $(\vartheta(\omega),0)$ as \emph{exceptional}.

In order to find systems with exceptional Lyapunov spectrum, we
restrict our attention to map cocycles generated by piecewise-affine
maps with a Markov partition.
We say an expanding map $T:I\to I$ is \emph{piecewise-affine} if there
is a partition $a_0=0<a_1<\ldots<a_m=1$ such that $T$ has constant
derivative on each interval $(a_{i-1},a_i)$.  Recall that, for a map
$T:I\to I$, a partition $\B=\{B_k\}_{k=1}^M$ of $I$ into intervals is
called a \emph{Markov partition} if for each pair $1\leq i,j \leq M$
such that $B_i\cap T(B_j)\neq \emptyset$, we have $B_i\subset T(B_j)$.
Associated to $T$ is a \emph{transition matrix}
$\Gamma=(\gamma_{i,j})_{1\leq i,j\leq M}$, where $\gamma_{i,j}=1$ if
$T(B_j)\supset B_i$ and $0$ otherwise. Given a partition, we denote
the set of partitioning points by $S_\B=I\backslash \bigcup_{B\in\B}
B$.

For a Markov partition $\B$, we let $\chi(\B)$ denote the space of
step-functions $I\to \mathbb{R}$ that are constant on the intervals of
$\B$.  We say $\B$ is a \emph{common Markov partition} for a
collection of maps $\{T_i\}_{i\in Z_K}$ if it is a Markov partition
for each map $T_i$. If the generators of the map cocycle have a common
Markov partition $\B$, then $\chi(\B)\subset \BV$ is an invariant
subspace for the associated Perron--Frobenius cocycle.

We now fix a collection $\T:=\{T_i\}_{i\in Z_K}$ of piecewise-affine
expanding maps with a common Markov partition $\B$. Clearly $\chi(\B)$
is an invariant subspace for the Perron--Frobenius cocycle.
Let $F$ be the quotient space $\BV/\chi(\B)$ with norm 
\mbox{$\|f\|_F:=\| f-Qf\|$}, where $Q:\BV\to\chi(\B)$ is the projection
$$
Q f(x)=\sum_{B\in\B} \alpha_B(f)\chi_B(x),\quad \alpha_B(f)=
\frac{1}{|B|}\int_B f(s)\ \mathrm{d}s.
$$
We identify an element $f+\chi(\B)\in F$ with the function $f-Qf$.
Thus $f\in F$ is characterised as having zero mean on each
interval $B\in\B$: that is, for each $B\in\B$, $\int_B f(s)\
\mathrm{d}s=0$.

The following Lemma bounds the growth rate for functions in $F$; we
will shortly see that $\vartheta(\omega)$ is the essential upper bound
for the Lyapunov spectrum of $\P(\omega)$.

\begin{lemma}\label{lem:Fbound}
Let $\Phi$ be a map cocycle generated by piecewise-affine expanding maps,
with a common Markov partition $\B$. Then 
\mbox{$\|\P^{(n)}(\omega)|_F\|\leq \sup_I \frac{3}{|D\Phi^{(n)}(\omega)|}$} 
for every $n\in\mathbb{N}$, and thus
$$
\sup_{f\in F} \lambda(\omega,f)\leq \vartheta(\omega).
$$
\end{lemma}
\begin{proof}
  The common Markov partition $\B$ allows for the following estimate
  of the variation of the Perron--Frobenius cocycle:
\begin{align*}\label{eqn:MarkovVarBound}
  \var_I \P^{(n)}(\omega) f &\leq \sum_{A\in\B} \var_I
  (\chi_A\cdot \P^{(n)}(\omega) f) \\
  &= \sum_{A\in\B} \var_{x\in I} \Bigg(\chi_A(x)\cdot
  \sum_{y\in [\Phi^{(n)}(\omega)]^{-1}(x)}
  \frac{f(y)}{|D\Phi^{(n)}(\omega)(y)|}\Bigg) \\
  &= \sum_{A\in\B} \var_{x\in I} \Bigg(
  \sum_{y\in [\Phi^{(n)}(\omega)]^{-1}(x)}
  \frac{ \chi_A(\Phi^{(n)}(\omega)(y))\cdot f(y)}
  {|D\Phi^{(n)}(\omega)(y)|}\Bigg) \\
  &= \sum_{A\in\B} \var_I \Bigg(\sum_{\substack{B\in \B^{(n)}(\omega)
  \\
  \Phi^{(n)}(\omega)B=A}} \frac{ \chi_A\circ \Phi^{(n)}(\omega)\cdot f}
  {|D\Phi^{(n)}(\omega)|}\Bigg) \\
  &\leq \sum_{A\in\B} \sum_{\substack{B\in \B^{(n)}(\omega)\\
   \Phi^{(n)}(\omega)B=A}} \var_I \Bigg(\frac{ \chi_B\cdot f}
   {|D\Phi^{(n)}(\omega)|}\Bigg) \\
\end{align*}
where $\B^{(n)}(\omega)=\bigvee_{i=0}^{n-1}\Phi^{(i)}(\omega)^{-1}\B$.
Thus
\begin{align}
  \var_I \P^{(n)}(\omega) f \leq \sum_{B\in\B^{(n)}(\omega)}
  \var_I \Bigg(\frac{\chi_B . f}{|D\Phi^{(n)}(\omega)|}\Bigg).
\end{align}
We show that $\var_I (f-Qf)\geq \|f-Qf\|_{L^1}$ for $f\in\textrm{BV}$,
from which it follows that \mbox{$\|f\|_F=\var_I(f-Qf)$}.
\begin{align}\label{eqn:L1VarIneq}
  \|f-Qf\|_{L^1} &=  \sum_{B\in\B} \int_B |(f-Qf)(s)|\ \mathrm{d}s \nonumber\\
  &\leq \sum_{B\in\B} |B| \left(\esssup_{x\in B} (f-Qf)(x) -
  \essinf_{x\in B} (f-Qf)(x)\right) \nonumber\\
  &\leq \sum_{B\in\B} \left(\esssup_{x\in B} (f-Qf)(x) -
  \essinf_{x\in B} (f-Qf)(x)\right) \nonumber\\
  &\leq \sum_{B\in\B} \var_B (f-Qf) \nonumber\\
  &\leq \var_I (f-Qf). %number
\end{align}

Let $\xi(f(x))=|\lim_{y\nearrow x} f(y)-\lim_{y\searrow x} f(y)|$
denote the \emph{jump} of $f$ at $x$.  Now, suppose $f\in F$. Applying
(\ref{eqn:MarkovVarBound}) we have
\begin{align}
  \|\P^{(n)}(\omega) f\| 
  &\leq \sum_{B\in\B^{(n)}(\omega)}
  \var_I \frac{\chi_B . f}{|D\Phi^{(n)}(\omega)|} \nonumber\\
  &\leq \sum_{B\in\B^{(n)}(\omega)} \var_B \frac{ f}{|D\Phi^{(n)}(\omega)|}
   + 2 \sum_{x\in S_{\B^{(n)}(\omega)}} \xi \left( \frac{f(x)}
   {|D\Phi^{(n)}(\omega)(x)|} \right) \nonumber\\
  &\leq 3\var_I \frac{f}{|D\Phi^{(n)}(\omega)|},\nonumber\\
  &\leq \sup_I \frac{3}{|D\Phi^{(n)}(\omega)|}. \var_I f,\nonumber\\
  &= \sup_I \frac{3}{|D\Phi^{(n)}(\omega)|}. \|f\|,
\end{align}
since $f\in F$ by (\ref{eqn:L1VarIneq}), giving the first part. Hence, 
for any $f\in F$,
$$
\lambda(\omega,f)\leq \limsup_{n\to\infty} \frac{1}{n}\log
\left(\frac{3}{|D\Phi^{(n)}(\omega)|}. \|f\|\right) = \vartheta(\omega)
$$
and the second result follows.
\end{proof}

We now prove that the exceptional Lyapunov spectrum of $\P$ is contained
in $\Lambda(\P(\omega)|_{\chi(\B)})$. For the autonomous case, see for
example \cite[Lemma 3.1]{BK98}.

\begin{proposition}\label{prop:QC}
Let $\Phi$ be a map cocycle generated by piecewise-affine expanding maps,
with a common Markov partition $\B$. Then the Lyapunov spectrum of the
Perron--Frobenius cocycle $\P(\omega)$ satisfies
$$
\Lambda(\P(\omega)) \subset \{x\leq \vartheta(\omega)\}\cup
\Lambda(\P(\omega)|_{\chi(\B)}),
$$
and thus $\P(\omega)$ has at most $\#\B$ exceptional Lyapunov
exponents.
\end{proposition}
\begin{proof}
Each $f\in\BV$ has a unique decomposition as a sum $f=f_{\chi(\B)}+f_F$,
where $f_{\chi(\B)}:=Q(f)\in\chi(\B)$ and $f_F:=f-Q(f)\in
F=\BV/\chi(\B)$. Notice that \mbox{$\lambda(\omega,f)\leq
\max\{\lambda(\omega,f_{\chi(\B)}),\lambda(\omega,f_F)\}$}. Thus, if
\mbox{$\lambda(\omega,f_{\chi(\B)})\leq \lambda(\omega,f_F)$}, then
we have
\mbox{$\lambda(\omega,f)\leq \vartheta(\omega)$} by Lemma \ref{lem:Fbound}.
Otherwise, we have \mbox{$\lambda(\omega,f_{\chi(\B)}) > \lambda(\omega,f_F)$},
in which case \mbox{$\lambda(\omega,f)= \lambda(\omega,f_{\chi(\B)})$}.
\end{proof}

\section{A stronger Multiplicative Ergodic Theorem for non-invertible matrices}

By Proposition \ref{prop:QC}, for each $\omega\in\Omega$, all
exceptional Lyapunov exponents of $\P(\omega)$ are contained in the
Lyapunov spectrum $\Lambda(\P(\omega)|_{\chi(\B)})$. We now represent
$\P(\omega)|_{\chi(\B)}$ as a matrix cocycle.

The set $\{\chi(B_i)\}_{i=1}^M$ forms a basis for $\chi(\B)$, and thus
each $f\in\chi(\B)$ may be written as $f=\sum_{i=1}^M v_i\chi_{B_i}$
in a unique way. Similarly, given $v\in\R^M$, we write $\left\langle
  v\right\rangle:=\sum_{i=1}^M v_i\chi_{B_i}$ for the corresponding
function in $\BV$.

For $T\in\T$, the matrix $P=(p_{i,j})_{1\leq i,j\leq M}$, where
$$
p_{i,j}=\frac{\gamma_{j,i}}{|DT|_{B_j}|}=\frac{m( T^{-1}(B_i)\cap
  B_j)}{m(B_j)},\quad 1\leq i,j\leq M,
$$
represents the Perron--Frobenius operator for $T$ with respect to the
basis $\{\chi(B_i)\}_{i=1}^M$ of $\chi(\B)$ (see, for example,
\cite[p.176]{BG97}).  That is, for each $v\in\R^M$ we have
$$
\P\left\langle v\right\rangle =\left\langle Pv\right\rangle.
$$

Let $P_i$ denote the matrix representing the restricted Perron--Frobenius
operator $\P_i|_{\chi(\B)}$ with respect to the basis
$\{\chi(B_i)\}_{i=1}^M$. The \emph{matrix cocycle} $A:\Z^+\times
\Omega\times \R^M\to \R^M$ is the one-sided cocycle with generator
$\tilde{A}(\omega)=P_{\omega_0}$.

Thus for each $\omega\in\Omega$, all exceptional Lyapunov exponents of
$\P(\omega)$ are captured by the Lyapunov spectrum of the cocycle
$\Lambda(A(\omega))=\Lambda(\P(\omega)|_{\chi(\B)})$.

The Multiplicative Ergodic Theorem for one-sided cocycles (see, for
example, \cite[Theorem 3.4.1]{A98}) provides us with a description of the
asymptotic behaviour of the matrix cocycle $A(\omega)$. It reveals
that the Lyapunov spectra $0=\lambda_1>\lambda_2>\cdots
\lambda_\ell\geq -\infty$ of $A(\omega)$ coincide for all $\omega$ in
a $\sigma$-invariant $\tilde{\Omega}\subset \Omega$ of full
$p$-measure. Moreover, it states that for each
$\omega\in\tilde{\Omega}$, a Lyapunov exponent $\lambda(\omega,v)$ of
$v\in\R^M$ for $A(\omega)$ is determined by the position of $v$ within
a \emph{flag} of nested subspaces
\mbox{$\{0\}=\V{\ell}{\omega}\subset\cdots\subset \V{2}{\omega}\subset
\V{1}{\omega}=\chi(\B)$}. Specifically, for each $i=1,\ldots,\ell$,
\begin{align}\label{eqn:flag}
  \lambda(\omega,v)=\lambda_i \Longleftrightarrow v\in \V{i}{\omega}
  \backslash\V{i+1}{\omega}.
\end{align}
In addition, the flag of subspaces is preserved by the action of the
cocycle: for $i=1,\ldots,\ell$,
$$
A(\omega) \V{i}{\omega}\subset \V{i}{\sigma\omega}.
$$
For two-sided matrix cocycles (see, for
example, \cite[Theorem 3.4.11]{A98}), by intersecting the corresponding
subspaces of the flags for the cocycle and for its inverse, one
obtains an \emph{Oseledets splitting}: that is, for each
$\omega\in\Omega$ we have a decomposition $\R^M=\bigoplus_{i=1}^\ell
\W{i}{\omega}$ such that for $i=1,\ldots,\ell$,
$$
v\in \W{i}{\omega}\setminus\{0\} \Longrightarrow
\lambda(\omega,v)=\lambda_i
$$
and
$$
A(\omega) \W{i}{\omega} = \W{i}{\sigma\omega}.
$$

Our matrix cocycle $A(\omega)$ sits between these two extremes: the 
shift $\sigma$ is invertible, but the matrices $\{P_i\}_{i\in Z_K}$
generating $A(\omega)$ are not. 
Because of the non-invertibility of the cocycle, we cannot use the 
standard approach described above to define an Oseledets splitting. 
The following new result relies on a push-forward approach to prove 
the existence of an Oseledets splitting even when the generators are 
non-invertible.

\begin{theorem3}
Let $\sigma$ be an ergodic invertible measure-preserving transformation
of the space $(\Omega,\mathcal B,\mu)$. Let $A\colon \Omega\to M_d(\R)$
be a measurable family of matrices satisfying
$$\int \log^+\|A(\omega)\|\,\mathrm{d}\mu(\omega)<\infty.
$$
Then there exist $\lambda_1>\lambda_2>\cdots>\lambda_\ell\geq -\infty$
and %%here
dimensions $m_1,\ldots,m_\ell$, with $m_1+\cdots+m_\ell=d$, and a
measurable family of subspaces $\W{j}{\omega}\subseteq \R^d$ such that
for almost every $\omega\in\Omega$ the following hold:
\begin{enumerate}
\item $\dim \W{j}{\omega}=m_j$;
  \item $\R^d=\bigoplus_{j=1}^\ell \W{j}{\omega}$;
  \item $A(\omega)\W{j}{\omega}\subseteq \W{j}{\sigma\omega}$ (with
    equality if $\lambda_j>-\infty$);
  \item for all $v\in \W{j}{\omega}\setminus\{0\}$, one has
    $$
    \lim_{n\to\infty}
    (1/n)\log\|A(\sigma^{n-1}\omega)\ldots A(\omega)v\|=\lambda_j.
    $$
  \end{enumerate}
\end{theorem3}
\begin{proof}
See Section \ref{sec:MET}
\end{proof}

\begin{remark}
It follows from part (iv) of Theorem 3 that we can determine the 
Lyapunov exponent for any vector $v\in\R^M\setminus \{0\}$ by 
$$
\lambda(\omega,v)=\lambda_i \Longleftrightarrow v\in
\bigoplus_{k=i}^{\ell+1} \W{k}{\omega} \setminus
\bigoplus_{k=i+1}^{\ell+1} \W{k}{\omega},
$$
where we set $\W{\ell+1}{\omega}=\{0\}$ for all $\omega\in\Omega$.
\end{remark}

We now apply Theorem \ref{thm:main} to our matrix cocycle $A(\omega)$
induced by $P(\omega)|_{\chi(\B)}$.
Consider the part of the Lyapunov spectrum of $A(\omega)$ that is
greater than $\vartheta(\omega)$. Let $r\leq \ell$ satisfy
$\lambda_{r+1}\leq\vartheta(\omega)< \lambda_{r}$. Thus, the part of
$\Lambda(\P(\omega)|_{\chi(\B)})$ strictly greater than
$\vartheta(\omega)$ is precisely
$\lambda_1>\lambda_2>\cdots>\lambda_r$.  It follows from Proposition
\ref{prop:QC} that the exceptional Lyapunov spectrum of $\P(\omega)$
is precisely $\lambda_2,\ldots,\lambda_r$.  By defining
$\mathcal{W}_i(\omega)=\{\langle v\rangle :v\in \W{i}{\omega}\}$ for
$1\leq i\leq r$, we transfer the splitting of $\R^M$ obtained from
Theorem \ref{thm:main} into a splitting of $\chi(\B)$ to get the
following result:

\begin{corollary}\label{cor:decomp}
Let $(\Omega,\mathcal{F},p,\sigma)$ be a measurable sequence space, and
$\Phi:\Z^+\times\Omega\times I\to I$ be the cocycle generated by
piecewise-affine expanding maps, with a common Markov partition $\B$.
Then there exists a forward invariant full measure subset
$\tilde{\Omega}\subset \Omega$,
$0=\lambda_1>\cdots>\lambda_{r}>\vartheta(\omega)$, and $m_1,\ldots,
m_r$, satisfying $m_1+\cdots+m_r\leq \#\B$, such that for all
$\omega\in\tilde{\Omega}$,
\begin{enumerate}
\item there exist subspaces $\mathcal{W}_{i}(\omega)\subset \BV$,
  $i=0,\ldots,r$, $\dim(\mathcal{W}_{i}(\omega))=m_i$;
\item $\P(\omega) \mathcal{W}_{i}(\omega)=\mathcal{W}_{i}(\sigma\omega)$;
\item $\langle v\rangle\in \mathcal{W}_{i}(\omega)\setminus\{0\} 
  \Longrightarrow \lambda(\omega,v)=\lambda_i$.
\end{enumerate}
\end{corollary}

Note that for periodic cocycles, $\tilde{\Omega}=\Omega$.

\section{Construction of periodic cocycles with exceptional Lyapunov
  spectrum}

In this section we build a periodic map cocycle for which the
Perron--Frobenius cocycle has an exceptional Lyapunov spectrum.

In \cite{DFS00} individual maps are constructed for which the
Perron--Frobenius operator has exceptional eigenvalues. The
construction uses so-called `almost-invariant' sets. Given a map
$T:I\to I$ with an absolutely continuous invariant probability measure
$\mu$, a subset $U\subset I$ is almost-invariant if
$$
\frac{\mu(U\cap T^{-1}U)}{\mu(U)}\approx 1.
$$
For a map with an almost-invariant set $U$, the transfer of mass
between $U$ and $I\backslash U$ is low, and so we expect to find that
a mean-zero function positive on $U$ and negative on $I\backslash U$
decays to zero slowly. It is shown that for piecewise-affine Markov
maps, one often obtains an almost-invariant set from the support
of either the positive or negative part of the eigenfunction associated 
to the second largest eigenvalue of the Perron--Frobenius operator. 

For this first example, we construct a cocycle over a periodic shift 
space of period $3$ that has a cyclic coherent structure.
More precisely, we take three maps, each having a distinct interval 
from the partition $\J=\{[0,1/3],[1/3,2/3],[2/3,1]\}$ of $S^1$ as 
an almost-invariant set. Post-composing these maps with the rotation 
by $1/3$, we form three new maps which we apply in sequence repeatedly, 
thus forming a periodic map cocycle $\Phi$.
In this way, each generator $\tilde{\Phi}(\omega)$ of the map cocycle 
moves the majority of the mass of one distinguished interval 
$J(\omega)\in\J$ into another interval $J(\sigma\omega)\in\J$ with 
some small dissipation. Thus a map $J:\Omega\to \J$ specifies the 
location of our coherent structure.

\begin{theorem}\label{thm:pereg}
There exist three piecewise-affine expanding maps
\mbox{$T_1,T_2,T_3:S^1\to S^1$} with a common Markov partition 
$\B$ that generate a map cocycle $\Phi:\Z^+\times\Omega\times S^1\to S^1$ 
over the sequence space $\Omega\subset Z_3^\mathbb{Z}$ generated by 
$\alpha=\overline{123}$ with the following properties for $i=1,2,3$:
\begin{enumerate}
\item each map $T_i$ preserves Lebesgue measure;
\item $\vartheta=\log 1/3$;
\item each finite dimensional restriction $\P_i|_{\chi(\B)}$ of the 
  Perron--Frobenius operator of $T_i$ has no exceptional eigenvalues;
\item $\P(\omega)$ has an exceptional Lyapunov spectrum that is 
independent of $\omega$ and satisfies
\begin{align*}
    \Lambda(\P)\cap(\vartheta,0)\supset\left\{
      \log\left(\frac{\sqrt[3]{8\pm 2\sqrt{11}}}{3}\right) \right\}.
\end{align*}
  \item the Oseledets subspaces
    $\mathcal{W}_{2}(\omega)$ corresponding to
    the largest exceptional Lyapunov exponent exists for all
    $\omega\in\Omega$, and depends only on $\omega_0$.
\end{enumerate}
\end{theorem}

For periodic map cocycles, one can find Lyapunov spectral points from
the eigenvalues of the cyclic composition of Perron--Frobenius
operators.

\begin{lemma}\label{lem:eta}
Consider a periodic map cocycle $\Phi:\Z^+\times\Omega\times I\to
\Omega\times I$ of period $R$. If $\eta$ is an
eigenvalue of the Perron--Frobenius operator $\P^{(R)}(\omega)$, then
$$
\frac{\log \eta}{R}\in\Lambda(\P(\omega)).
$$
\end{lemma}
\begin{proof}
  There exists a function $0\neq f\in\BV$ such that $\P^{(R)}(\omega)
  f= \eta f$. Hence for any $k\in\mathbb{N}$ and $0\leq r<R$,
$$
\min_{0\leq i<R}\{\|\P^{(i)}(\omega) f\|\}\leq\frac{\|\P^{(kR+r)}(\omega) f\|}{\eta^k}
\leq \max_{0\leq i<R}\{\|\P^{(i)}(\omega) f\|\},
$$
and the result follows.
\end{proof}

\begin{proof}[Proof of Theorem \ref{thm:pereg}]
Consider the partition $\J=\{J_1,J_2,J_3\}$ of $S^1$ into the subintervals
$J_i=[(i-1)/3,i/3]$. Let $\Phi:\Z^+\times\Omega\times I\to \Omega\times
I$ be the map cocycle with generator $\tilde{\Phi}(\omega)=T_{\omega_0}$,
where the maps $\T=\{T_1,T_2,T_3\}$ are given by
$$
T_i(x)=3x-\frac{j}{3}+\frac{G_{i,j}}{9}\mod{1}, \quad x\in
B_j=\left[\frac{j-1}{9},\frac{j}{9}\right), \quad j=1,\ldots,9,
$$
where
\begin{align*}
  G=\left(\begin{array}{ccccccccc}
      6 & 7 & 6 & 1 & 3 & 0 & 4 & 3 & 0 \\
      3 & 6 & 5 & 0 & 0 & 8 & 3 & 6 & 2 \\
      0 & 6 & 7 & 1 & 0 & 6 & 3 & 3 & 4 \\
\end{array}
\right).
\end{align*}
The graphs of $T_1,T_2,T_3$ are shown in Figure \ref{fig:T1T2T3}: note
that, by construction, each map $T_i$ largely maps the interval $J_i$
into the interval $J_{i+1}$, taking indices modulo $3$:
in fact, for $i\in {1,2,3}$,
$$
\frac{m(J_i\cap T^{-1}J_{i+1})}{m(J_i)}=\frac{8}{9}.
$$
Thus we have a coherent structure built around the family of intervals 
$J:\Omega\to\J$ given by $J(\omega)=J_{\omega_0}$. 

\begin{figure}[htbp]
    \centering
        \psfrag{T1}{$T_1$}
        \psfrag{T2}{$T_2$}
        \psfrag{T3}{$T_3$}
        \includegraphics[width=0.90\textwidth]{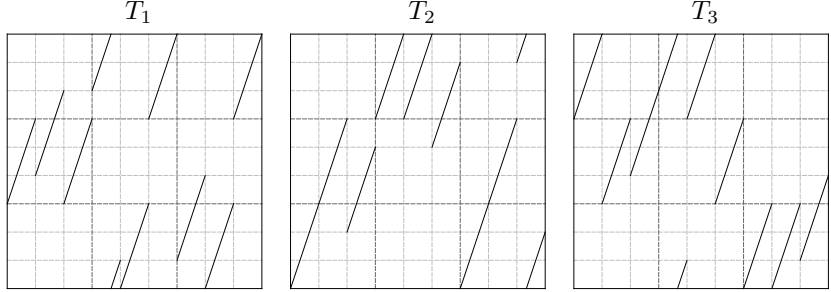}
        \caption{Graphs of $T_1,T_2,T_3$.}
    \label{fig:T1T2T3}
  \end{figure}
Note also that each map $T_i$ is piecewise-affine expanding and there is
a common Markov partition for $\T$ given by $\B=\{B_i:i=1,\ldots,9\}$.
Notice that for each map $T\in\T$ and interval $B\in\B$, $DT|_{B}=3$, and
so $\vartheta(\omega)=\log 1/3$ for each $\omega\in\Omega$.

Moreover, for each map $T\in\T$ and interval $B\in\B$, the preimage
$T^{-1}B$ has precisely three components, each of one third of the
length of $B$.  Thus each $T\in\T$ preserves Lebesgue measure, and
hence each $\Phi^{(n)}(\omega)$, $\omega\in\Omega$ and
$n\in\mathbb{N}$, does also.

As before, let $P_i$ denote the matrix of the restriction
$\P_i|_{\chi(\B)}$ with respect to the basis $\{\chi(B)\}_{i\in Z_9}$.
Here $P_i=\Gamma_{T_i}/3$ is the one third scaling of the transition
matrix $\Gamma_{T_i}$, which is itself easily observed from the graph
of $T_i$: the $(p,q)$th entry of the 0-1 matrix $\Gamma_i$ is $1$ if
and only if the graph of $T_i$ intersects the $(p,q)$th square of
$\B\times \B$.  Each matrix $P_i$ has a simple eigenvalue $1$, and all
other non-zero eigenvalues lie on the circle of radius $1/3$:
$$\begin{array}{l}
  \mathrm{spec}(P_1)=(1,-1/3,-1/3,0,\ldots,0)\\
  \mathrm{spec}(P_2)=(1,1/3,0,\ldots,0)\\
  \mathrm{spec}(P_3)=(1,-1/3,-1/6\pm\mathrm{i}\sqrt{3}/6,0,\ldots,0).\\
\end{array}
$$
Unlike in Theorem \ref{thm:nonpereg} in the following section, the 
maps used here cannot be expressed as different rotations of a single map.

We can find slowly decaying functions by examining the triple composition
\mbox{$\Phi^{(3)}(\alpha)=T_3\circ T_2\circ T_1$}. The
Perron--Frobenius operator $\P^{(3)}(\alpha)$,
when restricted to the
space $\chi(\B)$, can be represented by the matrix
$A^{(3)}(\alpha)=P_3P_2P_1$. We have
$$
\spec(A^{(3)}(\alpha))=\left( 1, \frac{2}{27}(4\pm
  2\sqrt{11}),0,\ldots,0 \right)
$$
Since the cocycle is periodic, we find that the spectrum of
$A^{(3)}(\omega)$ is independent of $\omega\in\Omega$.  Applying
Lemma \ref{lem:eta} we have that $\Lambda(\P)$ has the two exceptional
elements with approximate values
$$
\lambda_2\approx \log 0.8153,\quad \lambda_3\approx \log 0.3699.
$$
Moreover, these Lyapunov exponents are achieved by the corresponding
eigenvectors of $A^{(3)}(\omega)$. For $\omega=\alpha$, the space
$\W{2}{\alpha}$ is spanned by the second eigenvector $w_2$ of the matrix
$A^{(3)}(\alpha)=P_3P_2P_1$, with approximate entries
$$
w_2=(0.105, 0.193, 0.193, 0.008, -0.059, -0.059,-0.113, -0.134,
-0.134),
$$
and the graph of $\langle w_2\rangle\in \chi(\B)$, which spans
$\mathcal{W}_2(\alpha)$, is shown in Figure \ref{fig:V2}.
\begin{figure}
  \centering \psfrag{v2}{$\langle w_2\rangle$}
        \includegraphics[width=0.50\textwidth]{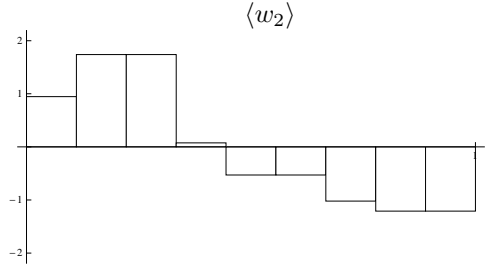}
        \caption{The graph of $\langle w_2\rangle\in \chi(\B)$ for Theorem 1.}
    \label{fig:V2}
  \end{figure}
For $i=1,2$, $\mathcal{W}_2(\sigma^i\alpha)$ is spanned by
$\langle w_2\rangle\circ\rho^{-i}$, where $\rho:S^1\to S^1$ is the 
rotation $\rho(x)=x+1/3\mod{1}$.
\end{proof}

Evidence of the cyclic coherent structure is visible in the second
eigenfunction of the Perron--Frobenius operator. Note that $J(\alpha)=[0,1/3]$
supports the majority of the mass of the positive part of $\langle w_2\rangle$. 
Similarly, the distinguished interval $J(\sigma^i\alpha)=[(i-1)/3,i/3]$, $i=1,2$, 
is picked up by  $\langle w_2\rangle\circ\rho^{-i}$.

\section{Construction of non-periodic cocycles with exceptional
  Lyapunov spectrum}

We now construct a non-periodic map cocycle with exceptional Lyapunov
spectrum. The map cocycle is generated by six maps, including $T_1$
used in the previous example. The shift space is taken to be a
subshift of finite type that has the Bernoulli shift on two symbols
$(Z_2^\Z,\theta,\mu)$ as a factor.

Let $\Theta\subset Z_6^\mathbb{Z}$ be the subshift of finite type
$$
\Theta:=\{\omega\in (Z_6)^\mathbb{Z}: \forall k\in\mathbb{Z},
(\omega_k,\omega_{k+1})=(i,j)\ \ \textrm{iff}\ \ E_{i,j}=1 \},
$$
with transition matrix
$$
E=(E_{i,j})_{1\leq i,j\leq 6}=\left(
\begin{array}{ccc|ccc}
  0 & 1 & 0 & 0 & 1 & 0 \\
  0 & 0 & 1 & 0 & 0 & 1 \\
  1 & 0 & 0 & 1 & 0 & 0 \\
  \hline
  0 & 0 & 1 & 0 & 0 & 1 \\
  1 & 0 & 0 & 1 & 0 & 0 \\
  0 & 1 & 0 & 0 & 1 & 0 \\
\end{array}
\right).
$$
We let $\sigma:\Theta\to\Theta$ be the left shift, and $p$ the uniform
measure on $\Theta$.  Notice that $(\{0,1\}^\Z,\theta,\mu)$ is a
factor of $(\Theta,\sigma,p)$ via the mapping
$$
h(\omega)_i= \left\{
\begin{array}{ll}
  0, & \omega_i\in\{1,2,3\}, \\
  1, & \omega_i\in\{4,5,6\}. \\
\end{array}
\right.
$$

The six maps are obtained from $T_1$ by rotations, and constructed so that
\begin{align}\label{eqn:nonperIso1}
  \frac{m(J_i\cap T^{-1}J_{i+1})}{m(J_i)} & = \frac{8}{9}\quad
   \mathrm{for}\: i=1,2,3,\\
  \frac{m(J_i\cap T^{-1}J_{i-1})}{m(J_i)} & = \frac{8}{9}\quad
  \mathrm{for}\: i=4,5,6.\label{eqn:nonperIso2}
\end{align}
From these maps we construct a map cocycle with a non-periodic coherent 
structure that is responsible for the slow decay.

\begin{theorem}\label{thm:nonpereg}
There exists a collection $\S$ of six piecewise-affine expanding maps
$S_1,\ldots,S_6:S^1\to S^1$ with a common Markov partition $\B$ that
generate a map cocycle $\Phi:\Z^+\times\Theta\times S^1\to  S^1$ with the
following properties for $i=1,\ldots,6$:
\begin{enumerate}
\item each map $S_i$ preserves Lebesgue measure;
\item $\vartheta=\log 1/3$;
\item the restricted Perron--Frobenius operator $\P_i|_{\chi(\B)}$ has
  no exceptional eigenvalues;
\item for each $\omega\in\Theta$, $\Lambda(\P(\omega))$ contains a 
unique exceptional exponent
$$
\log \frac{1+\sqrt{2}}{3}.
$$
\item there exists an Oseledets decomposition for all $\omega\in
  \Theta$, and the Oseledets subspace $\mathcal{W}_2(\omega)$ depends
  only on $\omega_0$.
\end{enumerate}
\end{theorem}
\begin{proof}
  Let $\rho:S^1\to S^1$ be the rotation $x\mapsto x+1/3\mod{1}$ and
  let $S:S^1\to S^1$ be the map given by
$$
S(x)=3x-\frac{j}{3}+\frac{g_{j}}{9}\mod{1}, \quad x\in
B_j=\left[\frac{j-1}{9},\frac{j}{9}\right), \quad j=1,\ldots,9,
$$
where $g=(3,4,3,7,0,6,1,0,6)$. The interval $J_1=[0,1/3]$ is an 
almost-invariant subset of $S^1$, with $m(J_1\cap S^{-1}J_1)/m(J_1)=8/9$. 
Let $P_S$ be the matrix of $\P_S|_{\chi(\B)}$ with respect to the basis
$\chi(\B)$. The spectrum of $P_S$ is
$$
\mathrm{spec}(P_S)=\left(1,\frac{1\pm \sqrt{2}}{3},0,\ldots,0\right).
$$

We define the collection of maps $\S=\{S_i\}_{i\in Z_6}$ in terms 
of $S$ and $\rho$:
\begin{align*}
  S_1 & =  \rho\circ S \\
  S_2 & =  \rho^2\circ S \circ \rho^2 \\
  S_3 & =  S \circ \rho \\
  S_4 & =  \rho^2\circ S \\
  S_5 & =  S \circ \rho^2 \\
  S_6 & =  \rho\circ S \circ \rho.
\end{align*}
\begin{figure}
  \centering
        \psfrag{S1}{$S_1$}
        \psfrag{S2}{$S_2$}
        \psfrag{S3}{$S_3$}
        \psfrag{S4}{$S_4$}
        \psfrag{S5}{$S_5$}
        \psfrag{S6}{$S_6$}
    \includegraphics[width=0.90\textwidth]{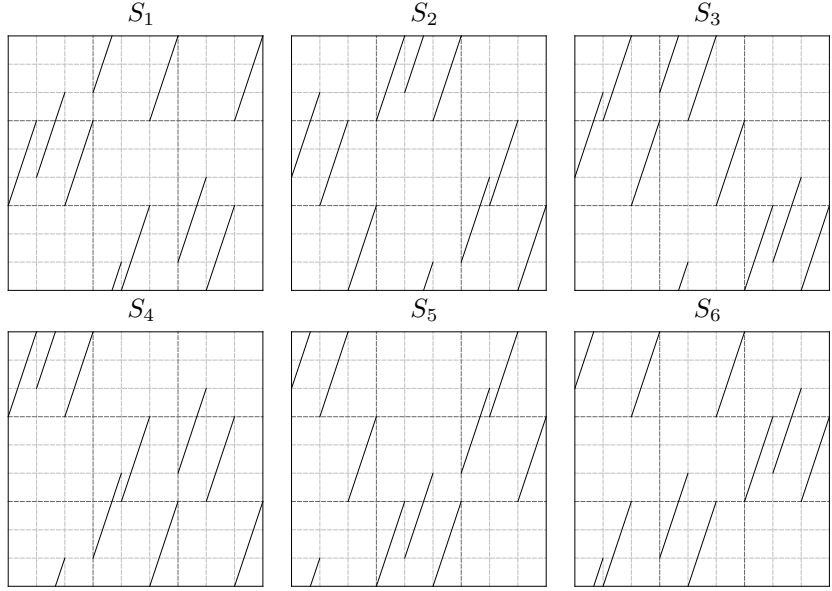}
                \caption{Graphs of $S_1,\ldots,S_6$.}
  \label{fig:S1toS6}
      \end{figure}
The graphs of $S_1,\ldots, S_6$ are shown in Figure \ref{fig:S1toS6}.
Note that the graph of $S_1$  is the same as that of $T_1$ shown in
Figure \ref{fig:T1T2T3}.

Let $\Phi:\Z^+\times \Theta\times S^1\to S^1$ be the map cocycle with
generator $\tilde{\Phi}(\omega)=S_{\omega_0}\in\S$. Let 
$\J=\{J_i\}_{i=1}^3$, where $J_i=[(i-1)/3,i/3]$. As a consequence 
of (\ref{eqn:nonperIso1}) and (\ref{eqn:nonperIso2}), we have a coherent
structure built around the family of intervals $J:\Theta\to\J$, where
\begin{align*}
J(\omega)=\left\{\begin{array}{ll}
   J_{\omega_0}, &\textrm{if}\ \omega_0\leq 3;\\
   J_{\omega_0-3}, &\textrm{if}\ \omega_0> 3.
   \end{array}\right.
\end{align*}
Let $\P_i$ be
the Perron--Frobenius operator of $S_i$. Let $\P:\Z^+\times
\Theta\times S^1\to S^1$ the Perron--Frobenius cocycle associated to
$\Phi$. Let $P_i$ be the matrix representing $\P_i|_{\chi(\B)}$ with
respect to the basis $\chi(\B)$ and let $A:\Z^+\times \Theta\times
S^1\to S^1$ be the matrix cocycle with generator
$\tilde{A}(\omega)=P_{\omega_0}$. Let $R$ denote the matrix with
$R_{i,j}=1$ if $i-j=3\mod{9}$ and $0$ otherwise. Note that $R^3$ is
the identity matrix. For $i=1,\ldots,6$, the formula for $P_i$ is
obtained directly from the formula for $S_i$ by replacing $\rho$ by
$R$ and replacing $S$ by $P_S$. Thus, for $i=1,\ldots,6$, we may write
$P_i=R^{l_i}.P_S. R^{r_i}$, where $l=(1,2,0,2,0,1)$ and
$r=(0,2,1,0,2,1)$.

One may confirm that
$$
\spec(P_i)=\left\{ \begin{array}{ll}
    (1, -1/3,-1/3,0,\ldots,0), &\textrm{if}\ i\leq 3;\\
    (1,0,\ldots,0), &\textrm{if}\ i> 3,
\end{array}\right.
$$
and so no map in $\S$ has exceptional eigenvalues.

Note that whenever $E_{i,j}=1$, we find $l_i+r_j=0\mod{3}$. Hence for
any $\omega\in\Theta$, we have that
$$
A^{(n)}(\omega)= R^{l(\omega_{n-1})}\circ (P_S)^n\circ
R^{r(\omega_{0})},
$$
with all inner $R$ factors cancelling.

Hence for any $v\in\R^M$,
\begin{align*}
  \|A^{(n)}(\omega)v \| & = \|R^{l_{\omega_{n-1}}}.(P_S)^n .R^{r_{\omega_0}}v\| \\
  & = \|(P_S)^n .R^{r(\omega_0)}v\| \\
  & = \|(P_S)^n .v'\|, \\
\end{align*}
where $v'=R^{r(\omega_0)}v$. So $\Lambda(A)$ is precisely the set of
logarithms of the eigenvalues of $P_S$, and in particular, is
independent of $\omega$. Thus, $\Lambda(\P)$ has a unique exceptional
exponent $\log (1+\sqrt{2})/3$ with approximate value $\log 0.8047$ 
for every $\omega\in\Theta$.

Let $w_2$ be an eigenvector of $P_S$ corresponding to the second
largest eigenvalue $(1+\sqrt{2})/3$.  The graph of $\langle
w_2\rangle\in \chi(\B)$, which spans $\mathcal{W}_2(\alpha)$, is shown
in Figure \ref{fig:V2nonper}.
\begin{figure}
  \centering \psfrag{v2}{$\langle w_2\rangle$}
        \includegraphics[width=0.50\textwidth]{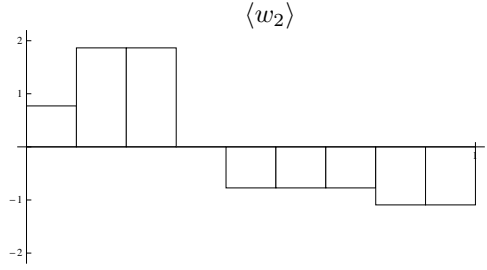}
        \caption{The graph of $\langle w_2\rangle\in \chi(\B)$ for Theorem 2.}
    \label{fig:V2nonper}
  \end{figure}
Moreover, we have an Oseledets splitting for every $\omega\in\Theta$: 
for each $\omega\in\Theta$, the function $\langle
R^{-r(\omega_0)}w_2\rangle$ spans the Oseledets  subspace
$\mathcal{W}_{2}(\omega)$ associated to $\log (1+\sqrt{2})/3$ and thus
$\mathcal{W}_{2}(\omega)$ depends only on $\omega_0$.
\end{proof}

As in the periodic example, the coherent structure
responsible for the slow decay is detected by the second eigenfunction of the
Perron--Frobenius operator.  When $\omega_0=1$, $J(\omega)=[0,1/3]$ is the
distinguished interval for $\Phi(\omega)$, and this interval supports the
majority of the mass of the positive part of the function $\langle
w_2\rangle$ spanning $\mathcal{W}_2(\omega)$.  More generally, for
$\omega\in\Theta$, the positive part of 
$\langle w_2\rangle\circ\rho^{-r(\omega_0)}=\langle R^{-r(\omega_0)} w_2\rangle$ 
is supported approximately on the interval $J(\omega)$.

\section{Numerical approximation of Oseledets subspaces}

In this section we outline a numerical algorithm to approximate the
$W_i(\omega)$ subspaces. The Oseledets splittings for the cocycles in 
Theorem 1 and Theorem 2 were explicitly constructed as eigenvectors.
In general, the Oseledets splittings are difficult to compute. The
algorithm is based on the push-forward limit argument developed in the
proof of Theorem \ref{thm:main}.

\begin{algorithm}[Approximation of the Oseledets subspaces
  $W_i(\omega)$ at $\omega\in\Omega$.]
\

\begin{enumerate}
\item Choose $M,N>0$ and form
$$
\Psi^{(M)}(\sigma^{-N}\omega):=(A^{(M)}(\sigma^{-N}\omega)^\mathrm{T}
A^{(M)}(\sigma^{-N}\omega))^{1/2M}
$$ as an approximation to the standard limiting matrix
$$
B(\sigma^{-N}\omega):=\lim_{M\to\infty}
\left(A^{(M)}(\sigma^{-N}\omega)^\mathrm{T}
  A^{(M)}(\sigma^{-N}\omega)\right)^{1/2M}
$$
appearing in the Multiplicative Ergodic Theorem.
\item Calculate the orthonormal eigenspace decomposition of
$\Psi^{(M)}(\sigma^{-N}\omega)$,
 denoted by $U_i^{(M)}(\sigma^{-N}\omega)$, $i=1,\ldots,\ell$.
\item Define
  $W_i^{(M,N)}(\omega):=A^{(N)}(\sigma^{-N}\omega)U_i^{(M)}(\sigma^{-N}\omega)$
  via the push forward under the matrix cocycle.
\item $W_i^{(M,N)}(\omega)$ is our numerical approximation to $W_i(\omega)$.
\end{enumerate}
\end{algorithm}

\begin{remarks}
\

\begin{enumerate}
\item Theorem \ref{thm:main} states that $W_i^{(\infty,N)}(\omega)\to W_i(\omega)$ as
$N\to\infty$.
\item This algorithm also provides an efficient numerical method for 
	calculating the Oseledets subspaces for two-sided linear cocycles.
\item There is freedom in the choice of relative sizes of $M$ and $N$: in
order to sample equal numbers of positive and negative terms of
$\omega$, we take $M=2N$.
\end{enumerate}
\end{remarks}

\begin{example}
To illustrate this technique, we calculate the 
Oseledets subspaces $\W{2}{\omega}$, $\omega\in\Theta$, 
for a non-periodic map cocycle, created from the
maps of Theorem \ref{thm:pereg} and the sequence space $\Theta$ of
Theorem \ref{thm:nonpereg}. Unlike the example of Theorem 2, this example does not
have a simple structure that makes it possible to relate the Oseledets
subspaces to those of a single autonomous transformation.

Let $\T=\{T_i\}_{i=1}^6$ denote the
collection of piecewise-affine expanding maps of the circle consisting
of the three maps $T_1,T_2,T_3$ defined in Theorem \ref{thm:pereg} and
the three maps $T_4=\rho\circ T_1$, $T_5=\rho\circ T_2$ and
$T_6=\rho\circ T_3$, where $\rho:S^1\to S^1$ is the rotation
$\rho(x)=x+1/3\mod{1}$ as before. The graphs of the maps in $\T$ are
shown in Figures \ref{fig:T1T2T3} and \ref{fig:T4T5T6}.  Let
$\Phi:\Z^+\times \Theta\times S^1\to S^1$ be the map cocycle generated
by $\T$. The collection $\T$ has a common Markov partition
$\B=\{[(i-1)/9,i/9]:i=1,\ldots,9\}$.
We expect to find an exceptional Lyapunov spectrum since the cocycle has 
a coherent structure similar to that of Theorem 2, built 
around the family of intervals $J:\Theta\to\J$ given by 
$$
J(\omega)=\left\{\begin{array}{ll}
   J_{\omega_0}, &\textrm{if}\ \omega_0\leq 3;\\
   J_{\omega_0-3}, &\textrm{if}\ \omega_0> 3.
   \end{array}\right.
$$

\begin{figure}
    \centering
    \psfrag{T4}{$T_4$}
    \psfrag{T5}{$T_5$}
    \psfrag{T6}{$T_6$}
        \includegraphics[width=0.90\textwidth]{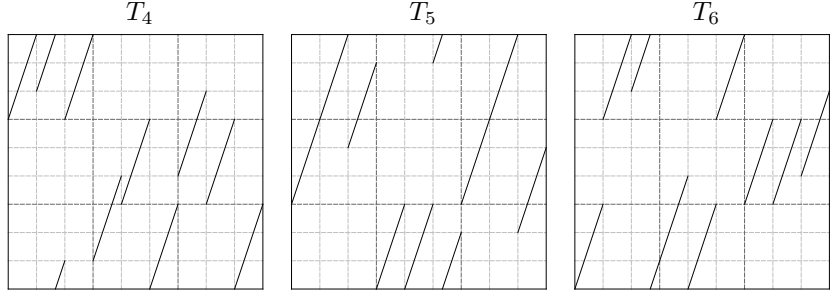}
        \caption{Graphs of $T_4$, $T_5$ and $T_6$.}
    \label{fig:T4T5T6}
  \end{figure}

We generate a test sequence in $\Theta$ as follows.  Let
$\hat{\alpha}^*\in \{0,1\}^\mathbb{N}$ be the fractional part of the
binary expansion of $\pi$:
$$
\hat{\alpha}^*=(0, 0, 1, 0, 0, 1, 0, 0, 0, 0, 1, 1, 1, 1, 1, 1, 0, 1, 1, 0,
1, 0, 1, 0, 1, 0, 0, \ldots),
$$
and extend it to a two-sided sequence $\alpha^*\in\{0,1\}^{\mathbb{Z}}$ 
by defining $\alpha^*_i=0$ for $i<0$.
We define $\omega^*=h^{-1}(\sigma^{120}\alpha^*)$, where $h$ is 
the $3$-to-$1$ factor defined in Section 6, and we take the 
inverse branch with $\omega^*_0=1$.
Note that $\omega^*\in\Theta$ has the form
$$
\omega^*=(\ldots,1,2,3,1,2,3,4,\ldots, 5, 4, 6, 2, 3, 1, 5, 4, 3, \underline{1}, 5,
1, 5, 4, 6, 2, 6, 5, 1,\ldots),
$$
where the zeroth term is underlined.

As before, we denote by $P_i$ the matrix representation of the
Perron--Frobenius operator $\P_i|_{\chi(\B)}$ of $T_i$,
$i=1,\ldots,6$, with respect the basis $\chi(\B)$, and denote by
\mbox{$A:\mathbb{Z}^+\times \Theta\times S^1\to S^1$} the matrix 
cocycle with the generator $\tilde{A}(\omega)=P_{\omega_0}$.
The Multiplicative
Ergodic Theorem states that for almost every $\omega$,
$\Psi^{(M)}(\omega)$ converges to a limit $B(\omega)$ as $M\to\infty$,
and moreover
$$
\Lambda(A)=\log \spec(B).
$$
Calculating $\Psi^{(M)}(\omega^*)$ for $M=40$, we find that 
$\Psi^{(M)}(\omega^*)$ has a simple eigenvalue $\lambda_2\approx 0.81$, 
suggesting that that $\P$ has exceptional Lyapunov exponent
approximately equal to $\log 0.81$.

In order to approximate the Oseledets subspace $\W{2}{\omega^*}$
numerically, we set \mbox{$M=2N=40$}, form the matrix $\Psi^{(2N)}(\sigma^{-N}\omega^*)$ and
denote by $u_2^{(2N)}(\sigma^{-N}\omega^*)$ the eigenvector corresponding
to the eigenvalue $\lambda_2$.  We then calculate
$$
A^{(N)}(\sigma^{-N}\omega^*) u_2^{(2N)}(\sigma^{-N}\omega^*)
$$
and normalize to give the vector $w_2^{(2N,N)}(\omega^*)$. The unit
vector $w_2^{(2N,N)}(\omega^*)$ is our approximation to a unit vector
spanning the subspace $W_2(\omega^*)$.

Although Theorem \ref{thm:main} holds only for a full $p$-measure subset
of $\Theta$, and so can tell us nothing about a particular sequence such
as $\omega^*$, we can still check whether its conclusions hold in this
case. Taking $N=20$, we calculate for $k=0,\ldots,7$, a vector 
$w_2^{(2N,N)}(\sigma^k\omega^*)$ spanning
$W_2^{(2N,N)}(\sigma^k\omega^*)$ (see Figure \ref{fig:8vectors}).
\begin{figure}[ht]
    \centering
    \psfrag{v1}[br][br]{$k=0$}
    \psfrag{v2}{$1$}
    \psfrag{v3}{$2$}
    \psfrag{v4}{$3$}
    \psfrag{v5}{$4$}
    \psfrag{v6}{$5$}
    \psfrag{v7}{$6$}
    \psfrag{v8}{$7$}
    \includegraphics[width=0.90\textwidth]{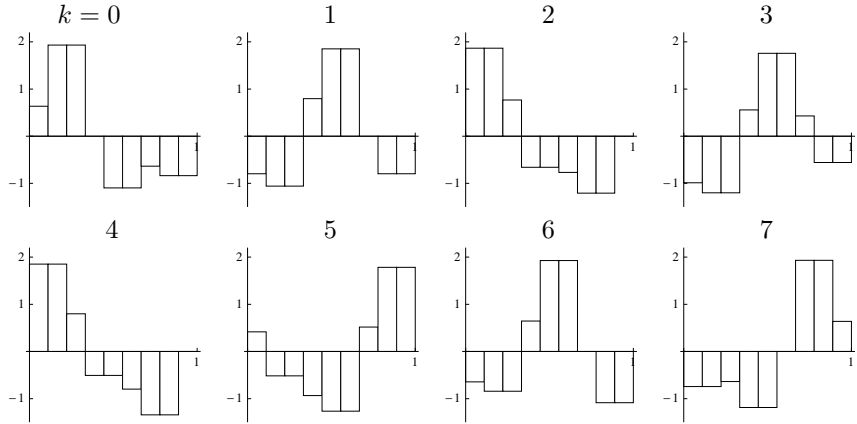}
    \caption{The graph of $\langle w_2^{(2N,N)}(\sigma^k\omega^*) \rangle$
    for $k=0,\ldots,7$.}
    \label{fig:8vectors}
  \end{figure}

Recall that $\{\omega^*\}_{i=0}^7=\{1,5,1,5,4,6,2,6\}$.
For $k=0,\ldots,7$, by examining Figure \ref{fig:8vectors}, and 
comparing with the list $(J(\sigma^k\omega^*))_{i=0}^7$ given by
\begin{align*}
([0,1/3],[1/3,2/3],[0,1/3],[1/3,2/3],[0,1/3],[2/3,1],[1/3,2/3],[2/3,1]),
\end{align*}
we see that the interval $J(\sigma^k\omega^*)$
is approximately picked up by the support of the positive part 
of $w_2^{(2N,N)}(\sigma^k\omega^*)$.

In order to check property (iii) of Theorem 3, that is, 
whether $A(\omega^*)W_2^{(2N,N)}(\omega^*)$ is close to 
$W_2^{(2N,N)}(\sigma^{k+1}\omega^*)$, we calculate the quantity
$$
\Delta^{(2N,N)}(\omega^*):= \min\left\{ \left\|\left\langle w_2^{(2N,N)}(\sigma\omega^*)\pm \frac{A(\omega^*)w_2^{(2N,N)}(\omega^*)}{\|\langle A(\omega^*)w_2^{(2N,N)}(\omega^*)\rangle\|_{L^1}}\right\rangle\right\|_{L^1} \right\},
$$
for $N=1,\ldots,20$ (see Figure \ref{fig:Logplot}).

\begin{figure}[h]
	\centering
		\includegraphics[width=0.90\textwidth]{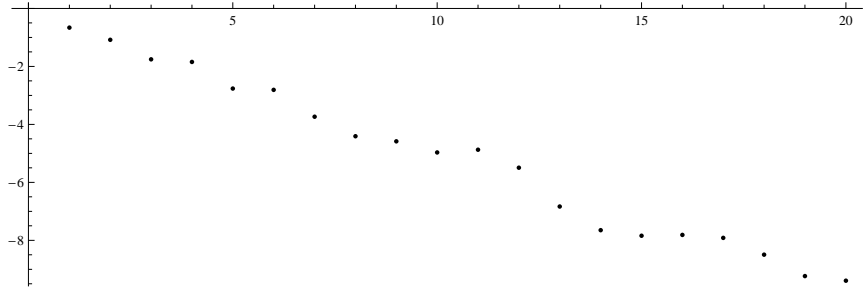}
	\caption{Graph showing $\log_{10} \Delta^{(2N,N)}(\omega^*)$ against $N$ for $N=1,\ldots,20$.}
	\label{fig:Logplot}
\end{figure}
Thus for $N=20$, there are unit $L^1$-norm functions
spanning the $\mathcal{W}_2^{(2N,N)}(\sigma\omega^*)$ and 
$\P(\omega^*)\mathcal{W}_2^{(2N,N)}(\omega^*)$ 
subspaces whose difference in $L^1$-norm is less than $10^{-8}$.

Recall that for the cocycle in Theorem 2, the Oseledets subspace $W_2(\omega)$ 
in fact independent of $\omega_i$ for $i\neq 0$. This contrasts with the current example: 
to see that here the Oseledets spaces  $W_2(\omega)$ do not depend only on $\omega_0$, 
it is enough to observe, for example, that $\omega_0=\omega_2=1$ but 
$w_2^{(2N,N)}(\omega^*)$ and $w_2^{(2N,N)}(\sigma^2\omega^*)$ are markedly dissimilar.
\end{example}

\section{Proof of the Multiplicative Ergodic Theorem for
  non-invertible matrices}\label{sec:MET}

In this section we present a strengthened version of the Multiplicative
Ergodic Theorem (MET) for the case of non-invertible matrices. Let
$\sigma$ be an invertible measure-preserving transformation of
$(\Omega,\mathcal B,\mu)$ and consider a linear cocycle
$P:\mathbb{Z^+}\times \Omega\times \R^d\to \R^d$. Note that even though
the matrices may be non-invertible, the invertibility of $\sigma$ is
crucial to the argument. If the matrices are invertible then the
two-sided cocycle is naturally defined as a map
$P\colon\mathbb{Z}\times\Omega\times \R^d\to\R^d$.

Recall that in the case of a one-sided linear cocycle (satisfying certain
integrability conditions), the MET provides an invariant flag of
subspaces of $\R^d$ characterising the exponential growth rates of all
vectors. For a two-sided cocycle, one obtains an invariant splitting of
$\R^d$ into {\sl Oseledets spaces} by considering the intersection of
each subspace in the flag of the forward cocycle with the corresponding
subspace of the flag of the backward cocycle. Non-zero vectors $v$ in the
$j$th Oseledets space $W_j(\omega)$ satisfy
$\lim_{n\to\pm\infty}(1/n)\log\|P(n,\omega,v)\|\to\lambda_j$.

In the case of a one-sided cocycle it clearly makes no sense to consider 
the limit $\lim_{n\to-\infty}(1/n)\log\|P(n,\omega,v)\|$. 
Nevertheless one may still
hope for an invariant splitting of $\R^d$ rather than an invariant flag.
This distinction is important if one is interested in the \emph{vector} 
corresponding to the one
of the top characteristic exponents: the flag would only provide an invariant
family of \emph{high-dimensional subspaces} with the property that most vectors
in the space have the correct expansion rate, whereas a splitting would
provide an invariant family of \emph{low-dimensional subspaces}, whose vectors
are responsible for all expansion at the chosen rate.

In this section we obtain a decomposition into Oseledets subspaces for a
one-sided forward cocycle over an invertible transformation by means of a
push-forward limit argument.

Let $\|\cdot\|$ denote the matrix operator norm with respect to the
Euclidean norm on $\R^d$.

\begin{theorem}\label{thm:main}
Let $\sigma$ be an invertible ergodic measure-preserving transformation
of the space $(\Omega,\mathcal B,\mu)$. Let $A\colon \Omega\to M_d(\R)$
be a measurable family of matrices satisfying
$$\int \log^+\|A(\omega)\|\,\mathrm{d}\mu(\omega)<\infty.
$$
Then there exist $\lambda_1>\lambda_2>\cdots>\lambda_\ell\geq -\infty$
and %%here
dimensions $m_1,\ldots,m_\ell$, with $m_1+\cdots+m_\ell=d$, and a
measurable family of subspaces $\W{j}{\omega}\subseteq \R^d$ such that
for almost every $\omega\in\Omega$ the following hold:
\begin{enumerate}
\item $\dim \W{j}{\omega}=m_j$;
  \item $\R^d=\bigoplus_{j=1}^\ell \W{j}{\omega}$;
  \item $A(\omega)\W{j}{\omega}\subseteq \W{j}{\sigma\omega}$ (with
    equality if $\lambda_j>-\infty$);
  \item for all $v\in \W{j}{\omega}\setminus\{0\}$, one has
   $$
   \lim_{n\to\infty}
   (1/n)\log\|A(\sigma^{n-1}\omega)\ldots A(\omega)v\|= \lambda_j.
   $$
  \end{enumerate}
\end{theorem}

\begin{lemma}\label{lem:frame}
  Let $B\colon\Omega\to M_d(\R)$ be a measurable mapping into the space of
  symmetric matrices such that for almost all $\omega$, $B(\omega)$
  has eigenvalues $\mu_1,\ldots,\mu_\ell$ with multiplicities
  $m_1,\ldots,m_\ell$. Then there exists a measurable
  family $\left(e^j_i(\omega)\right)_{1\le j\le \ell,\ 1\le i\le m_j}$ 
  of vectors such that the $\left(e^j_i(\omega)\right)$ form an
  orthonormal basis of $\R^d$ and $e^j_i(\omega)$ lies in the
  $\mu_j$ eigenspace of $B(\omega)$.
\end{lemma}

\begin{proof}
  Consider the map $R$ that takes a matrix and applies a single step of a
  row-reduction algorithm (e.g.~find the first column that is not in
  row-reduced echelon form; transpose rows to put a non-zero entry in
  the correct place; divide so the leading coefficient is 1; subtract
  multiples of that row from all of the others; repeat) or does
  nothing in the case that the matrix is already in row-reduced
  echelon form. 
  The domains of the pieces are measurable and therefore $R$ is measurable.
  For all matrices $A$, $R^n(A)$ is a
  convergent sequence so the limit $RRE(A)$ is a measurable function
  of the matrix.

  A collection of vectors spanning the kernel of a row-reduced matrix
  may be obtained in a measurable way.
  % by systematically letting one
  %of the free variables take the value 1 while the other free
  %variables take the value 0.
  These vectors may then be measurably converted to an orthonormal set
  by applying the Gram--Schmidt orthogonalization algorithm.

  We apply this by taking a symmetric matrix $B$ with eigenvalues
  $\mu_1,\ldots,\mu_\ell$ with multiplicities $m_1,\ldots,m_\ell$.
  We find an orthogonal set of vectors with each of the eigenvalues by
  applying the above procedures to $B-\mu_j I$. Since all
  operations are measurable the proof is complete.
\end{proof}

\begin{lemma}\label{lem:subrev}
  Let $\sigma\colon \Omega\to\Omega$ be an invertible ergodic
  measure-preserving transformation and let $(f_n)_{n=1}^\infty$ be a
  subadditive sequence of functions (that is a sequence such that for
  every $\omega\in\Omega$ and each $m$ and $n$, $f_{n+m}(\omega)\le
  f_n(\omega)+ f_m(\sigma^n\omega)$). Assume further that
  $\max(f_1,0)$ is an $L^1$ function.  Then there is a
  $C\in[-\infty,\infty)$ such that for almost every $\omega$ one has
  $f_n(\omega)/n\to C$ and $f_n(\sigma^{-n}\omega)/n\to C$.
\end{lemma}

\begin{proof}
  The fact that there is a $C$ such that $f_n/n\to C$ is Kingman's
  subadditive ergodic theorem. Letting $g_n(\omega)=f_n(\sigma^{-n}\omega)$,
  we see that $g_{n+m}(\omega)\le g_n(\omega)+g_m(\sigma^{-n}\omega)$ so that
  the subadditive ergodic theorem applies to $g_n$ also (with the
  measure-preserving transformation being $\sigma^{-1}$) and there is a
  constant $D$ such that $g_n(\omega)/n\to D$ for almost all $\omega$.

  Since $f_n/n$ converges pointwise to $C$ it also converges to $C$ in
  measure. Similarly $g_n/n$ converges in measure
  to $D$.  Since $f_n/n$ and $g_n/n$ have the same distribution, the
  constants to which they converge in measure must be equal.
\end{proof}

\begin{lemma}\label{lem:revexps}
  Let $\sigma\colon \Omega\to\Omega$ be an invertible ergodic
  measure-preserving transformation and let $A\colon\Omega\to M_d(\R)$
  be a measurable family of matrices satisfying
  \begin{align*}
  \int\log^+\|A(\omega)\|\,\mathrm{d}\mu(\omega)<\infty.
  \end{align*}
  Let $S$ be
  the multiset of characteristic exponents. Given $\omega\in\Omega$,
  let $SV^{(n)}(\omega)$ be the multiset of logarithms of the
  $n$th roots of the
  singular values of $A^{(n)}(\sigma^{-n}\omega)$. Then for almost
  every $\omega$, $SV^{(n)}(\omega)\to S$.
\end{lemma}

\begin{proof}
  Consider the family $\omega\mapsto A^\mathrm{T}(\omega)$ with respect to the
  dynamical system $\sigma^{-1}$. Let the characteristic exponents be
  the multiset $S'$. This means that letting $SV'^{(n)}(\omega)$ be
  the multiset of $n$th roots of singular values of
  $A^\mathrm{T}(\sigma^{-n}\omega)\ldots A^\mathrm{T}(\sigma^{-1}\omega)$, one has
  $SV'^{(n)}(\omega)\to S'$ for almost every $\omega\in\Omega$ by the MET.
  Since singular values are preserved by
  taking transposes we see that $SV'^{(n)}(\omega)=SV^{(n)}(\omega)$.
  Thus it suffices to prove $S=S'$. To see this, note that
  $(1/n)\log\|\bigwedge^kA^{(n)}(\omega)\|$ converges to the sum of
  the largest $k$ members of $S$, and \mbox{$(1/n)\log\|\bigwedge^k
  A^\mathrm{T}(\sigma^{-n}\omega)\ldots A^\mathrm{T}(\sigma^{-1}\omega)\|$} converges to
  the sum of the largest $k$ members of $S'$, but these limits are
  equal by Lemma \ref{lem:subrev}.
\end{proof}

\begin{proof}[Proof of Theorem \ref{thm:main}]
	In the course of the proof we shall repeatedly use the symbol $C$ to denote 
	various constants depending only on $\omega$.

  We write $A^{(n)}(\omega)$ for the matrix product
  $A(\sigma^{n-1}\omega)\ldots A(\omega)$. From standard proofs of
  the MET, we have that
  $[A^{(n)}(\omega)^\mathrm{T}A^{(n)}(\omega)]^{1/(2n)}$ is convergent to a
  positive semi-definite matrix $B(\omega)$, for almost all $\omega$,
  with eigenvalues
  $e^{\lambda_1}>\ldots>e^{\lambda_\ell}$ with the correct multiplicities.
  We therefore let $(e^j_i(\omega))$ be as in Lemma \ref{lem:frame} and
  let $\U{j}{\omega}$ be the subspace of $\R^d$ spanned by
  $\{e^j_i(\omega)\colon 1\le i\le m_j\}$. The standard proofs of the
  MET show that
  if one lets $\V{j}{\omega}=\bigoplus_{i=j}^\ell \U{i}{\omega}$
  then the vector spaces $\V{j}{\omega}$ satisfy
  \begin{enumerate}
  \item $A(\omega) \V{j}{\omega}\subseteq \V{j}{\sigma\omega}$;
  \item For all $v\in \V{j}{\omega}\setminus \V{j+1}{\omega}$,
    $\lim_{n\to\infty}(1/n)\log \|A^{(n)}(\omega)v\|\to
    \lambda_j$;
  \end{enumerate}

  For $j<\ell$, let $\Wpush{j}{\omega}{n}=A^{(n)}(\sigma^{-n}\omega)
  \U{j}{\sigma^{-n}\omega}$ and let
  $\W{\ell}{\omega}=\U{\ell}{\omega}$. Then we claim the following:
    \begin{enumerate}
    \item \label{it:conv}For $j<\ell$, $\Wpush{j}{\omega}{n}$
      converges to an $m_j$-dimensional subspace $\W{j}{\omega}$;
      \item \label{it:inv}$A(\omega)\W{j}{\omega}\subseteq \W{j}{\sigma\omega}$;
      \item \label{it:growth}If $x\in \W{j}{\omega}\setminus\{0\}$,
        then $(1/n)\log\|A^{(n)}(\omega)x\|\to\lambda_j$.
      \item \label{it:dsum}$\V{j+1}{\omega}\oplus
        \W{j}{\omega}=\V{j}{\omega}$.
    \end{enumerate}
    Notice that
    $\Wpush{j}{\sigma\omega}{n+1}=A(\omega)\Wpush{j}{\omega}{n}$ so
    that in the case $j<\ell$, \eqref{it:inv} follows from
    \eqref{it:conv} and the definition. For $j=\ell$, \eqref{it:inv}
    and \eqref{it:growth} follow from the standard MET proofs.

    Fix a $j<\ell$ and consider a basis $B_0(\omega)=\{e_k^i(\omega)
    \colon k> j,\ i\le m_k\}$ for $\V{j+1}{\omega}$ and a basis
    $B_1(\omega)=\{e_j^i(\omega)\colon i\le m_j\}$ for
    $\U{j}{\omega}$.
    The union of $B_0(\omega)$ and $B_1(\omega)$
    gives an orthonormal basis for $\V{j}{\omega}$. Since
    $A(\omega)\V{j+1}{\omega}\subset \V{j+1}{\sigma\omega}$ and
    $A(\omega)\V{j}{\omega}\subset \V{j}{\sigma\omega}$, it follows
    that if we express the linear transformation represented by
    $A(\omega)$ with respect to the bases $B_1(\omega)\cup
    B_0(\omega)$ and $B_1(\sigma\omega)\cup B_0(\sigma\omega)$, the
    matrix is of the form
    $$
    L(\omega)=\begin{pmatrix}
      A_{11}(\omega)&0\\
      A_{10}(\omega)&A_{00}(\omega)
    \end{pmatrix},
    $$
  where if $V_{j+1}(\omega)$ is of dimension $q=m_{j+1}+\ldots+m_\ell$,
  the matrices $A_{11}(\omega)$, $A_{10}(\omega)$ and $A_{00}(\omega)$ have
  dimensions $m_j\times m_j$, $q\times m_j$ and $q\times q$ respectively.

  By definition, $L^{(n)}(\omega)=L(\sigma^{n-1}\omega)\ldots L(\omega)$.
  By analogy with the above we name the components of this matrix as
  follows:
  $$
    L^{(n)}(\omega)=
    \begin{pmatrix}
      A_{11}^{(n)}(\omega)&0\\
      A_{10}^{(n)}(\omega)&A_{00}^{(n)}(\omega).
    \end{pmatrix}
  $$

  We will need the following matrix identities:

  \begin{claim}\label{label:claimtwo}
  With $A_{ij}^{(n)}$ defined as above we have
  \begin{align}
    A_{11}^{(n)}(\omega)&=A_{11}(\sigma^{n-1}\omega)\ldots A_{11}(\omega)\\
    A_{00}^{(n)}(\omega)&=A_{00}(\sigma^{n-1}\omega)\ldots A_{00}(\omega)\\
    A_{10}^{(n)}(\omega)&=\sum_{k=0}^{n-1}
    A_{00}^{(k)}(\sigma^{n-k}\omega)A_{10}(\sigma^{n-k-1}\omega)
    A_{11}^{(n-k-1)}(\omega).\label{eq:crossterms}
  \end{align}
\end{claim}
\begin{proof}
  The first two equalities are immediate and the third follows by
  induction on $n$.
\end{proof}

\begin{claim}\label{label:Kingman}
  For almost every $\omega\in\Omega$,
  $\log \|A^{(n)}_{00}(\omega)\|\to\lambda_{j+1}$ as $n\to\infty$.
\end{claim}
\begin{proof}
  One has for each $i>j$ and $1\le k\le m_i$,
  $(1/n)\log \|A^{(n)}(\omega)e^i_k\|\to\lambda_i$, by the MET.
  It follows that considering $A^{(n)}(\omega)$ as a linear map on
  $\V{j+1}{\omega}$,
  $(1/n)\log\|A^{(n)}(\omega)\vert_{\V{j+1}{\omega}}\|\to\lambda_{j+1}$.
  Thus
  \begin{align*}
  (1/n)\log\|A_{00}^{(n)}(\omega)\|\to\lambda_{j+1}.
  \end{align*}
\end{proof}

\begin{claim}\label{claim:backwards}
For every $\epsilon>0$ and for almost every $\omega\in\Omega$, there is
$D_1(\omega)$ such that $\|A_{00}^{(n)}(\sigma^{-n}\omega)\|\le
D_1(\omega)e^{n(\lambda_{j+1}+\epsilon)}$ for all $n\geq 0$.
\end{claim}
\begin{proof}
  Let $f_n(\omega)=\log \|A_{00}^{(n)}(\omega)\|$.  This is a
  sub-additive sequence of functions and $f_n(\omega)/n\to\log
  \lambda_{j+1}$ for almost every $\omega$ by Claim \ref{label:Kingman}.
  Applying Lemma \ref{lem:subrev} we see that
  $f_n(\sigma^{-n}\omega)/n\to\log \lambda_{j+1}$ for almost every
  $\omega$. The claim follows.
\end{proof}

\begin{claim}\label{claim:Birkhoff}
  For every $\epsilon>0$ and for almost every $\omega\in\Omega$, there is
  a $D_2(\omega)<\infty$ such that for all $n\geq 0$ one has
  $\|A_{10}(\sigma^{-n}\omega)\|\le D_2(\omega)e^{\epsilon n}$.
\end{claim}
\begin{proof}
  By hypothesis $\log\|A(\omega)\|$ is an integrable function and hence
  by a standard corollary of Birkhoff's theorem one has
  $\log\|A(\sigma^{-n}\omega)\|/n\to 0$. It follows that 
  $\|A(\sigma^{-n}\omega)\|\le D_2(\omega)e^{\epsilon n}$ for a 
  suitable $D_2(\omega)$. Since
  $\|A_{10}(\omega)\|\le \|A(\omega)\|$ the result follows.
\end{proof}

\begin{claim}\label{claim:claimone}
  Under the above conditions,
  $\left({A_{11}^{(n)}(\omega)}^\mathrm{T}A_{11}^{(n)}(\omega)\right)^{1/(2n)}
  \longrightarrow e^{\lambda_j} I_{m_j}$.
\end{claim}
\begin{proof}
To see this it is sufficient to show that every non-zero vector in
$\U{j}{\omega}$ has growth rate $\lambda_j$. Let $u\in \U{j}{\omega}$ have
expansion $u=\sum_{i\le m_j} v_i e^j_i(\omega)$.

First we show that $A_{10}^{(n)}(\omega)v$ doesn't grow any faster than
$A_{11}^{(n)}(\omega)v$. Note that
$\|A^{(n)}(\omega)u\|^2=\|A_{11}^{(n)}(\omega)v\|^2+
\|A_{10}^{(n)}(\omega)v\|^2$ so that we have
$\|A_{10}^{(n)}(\omega)v\|+\|A_{11}^{(n)}(\omega)v\|$ grows at rate
$\lambda_{j}$. Applying the MET to $A_{11}^{(n)}(\omega)$, we see that
$A_{11}^{(n)}(\omega)v$ grows at some rate $\Lambda$. We will show that
$A_{10}^{(n)}(\omega)v$ grows at a rate no greater than
$\max(\Lambda,\lambda_{j+1})$. It will follow that $\Lambda=\lambda_{j}$.

Equality \eqref{eq:crossterms} gives
\begin{eqnarray*}
  \|A_{10}^{(n)}(\omega)v\|\le \sum_{k=0}^{n-1}
  \|A_{00}^{(k)}(\sigma^{n-k}\omega)\|\|A_{10}(\sigma^{n-k-1}\omega)\|
  \|A_{11}^{(n-k-1)}(\omega)v\|
\end{eqnarray*}
Fix an arbitrary $\epsilon>0$. Claim \ref{claim:backwards} shows that
$\|A_{00}^{(k)}(\sigma^{n-k}\omega)\|\le
D_1(\sigma^n\omega)e^{k(\lambda_{j+1}+\epsilon)}$. Using Claim
\ref{claim:Birkhoff} also, we see
\begin{eqnarray*}
  \|A_{10}^{(n)}(\omega)v\|\le
  D_1(\sigma^n\omega)D_2(\omega)\sum_{k=0}^{n-1}e^{k(\lambda_{j+1}+\epsilon)}
  e^{\epsilon(n-k-1)}e^{(\Lambda+\epsilon)(n-k-1)}.
\end{eqnarray*}
There exists $M$ such that $D_1(\omega)<M$ on a positive measure 
subset of $\Omega$. By the ergodicity of $\sigma$, there are 
infinitely many $n$ for which \mbox{$D_1(\sigma^n\omega)<M$}.  
For these $n$, the right hand side of the
inequality is bounded above by
$D_2(\omega)Mne^{n(\max(\lambda_{j+1},\Lambda)+2\epsilon)}$. It follows that 
\mbox{$\liminf\log \|A_{10}^{(n)}(\omega)v\|^{1/n}\le
\max(\lambda_{j+1},\Lambda)$} and thus
$$
\liminf\log\|A^{(n)}(\omega)u\|^{1/n}\le \max(\lambda_{j+1},\Lambda).
$$
Since on the other hand \mbox{$\lim\log\|A^{(n)}(\omega)u\|^{1/n}=\lambda_{j}$},
we conclude that $\Lambda\ge\lambda_j$. Since \mbox{$\|A^{(n)}(\omega)u\|\ge
\|A^{(n)}_{11}(\omega)v\|$} we have $\lambda_j\ge\Lambda$ so that
$\Lambda=\lambda_{j}$ as required.
\end{proof}

We now estimate
\begin{eqnarray*}
  g_n(\omega)=\max_{v\in S_1}\frac{\|A_{10}^{(n)}(\sigma^{-n}\omega)v\|}
  {\|A_{11}^{(n)}(\sigma^{-n}\omega)v\|},
\end{eqnarray*}
where $S_1$ denotes the unit sphere in $\R^{m_j}$. Note that by
scale-invariance one could equivalently define $g_n$ by taking the
maximum over $\R^{m_j}\setminus\{0\}$.

We have
\begin{align*}
  g_n(\omega)&=\max_{v\in S_1}\frac
  {\left\|\sum_{k=0}^{n-1}A_{00}^{(k)}(\sigma^{-k}\omega)A_{10}
      (\sigma^{-(k+1)}\omega)
      A_{11}^{(n-k-1)}(\sigma^{-n}\omega)v\right\|}
  {\left\|A_{11}^{(n)}(\sigma^{-n}\omega)v\right\|}\\
  &\le \sum_{k=0}^{n-1}\max_{v\in
    S_1}\frac{\left\|A_{00}^{(k)}(\sigma^{-k}\omega)
      A_{10}(\sigma^{-(k+1)}\omega)
      A_{11}^{(n-k-1)}(\sigma^{-n}\omega)v\right\|}
  {\left\|A_{11}^{(k+1)}(\sigma^{-(k+1)}\omega)
      A_{11}^{(n-k-1)}(\sigma^{-n}\omega)v\right\|}\\
  &=\sum_{k=0}^{n-1}\max_{u\in
    S_1}\frac{\left\|A_{00}^{(k)}(\sigma^{-k}\omega)
      A_{10}(\sigma^{-(k+1)}\omega)u\right\|}
  {\left\|A_{11}^{(k+1)}(\sigma^{-(k+1)}\omega)u\right\|}\\
  &\le\sum_{k=0}^{n-1} \frac{\max_{u\in
      S_1}\left\|A_{00}^{(k)}(\sigma^{-k}\omega)
      A_{10}(\sigma^{-(k+1)}\omega)u\right\|} {\min_{u\in
      S_1}\left\|A_{11}^{(k+1)}(\sigma^{-(k+1)}\omega)u\right\|}.
\end{align*}

Note that in the third line we are making use of the fact that
$A_{11}^{(n-k-1)}(\sigma^{-n}\omega)$ is invertible.

Let $\epsilon<(\lambda_j-\lambda_{j+1})/4$ be fixed for the remainder of
the proof. By Lemma \ref{lem:revexps} and Claim \ref{claim:claimone} the
$k$th roots of the singular values of $A_{11}^{(k)}(\sigma^{-k}\omega)$
all converge to $e^{\lambda_{j}}$.  It follows that there is a $C>0$
depending on $\omega$ such that for all $k$,
\begin{eqnarray}\label{eq:A11lower}
  \min_{u\in
    S_1}\left\|A_{11}^{(k+1)}(\sigma^{-(k+1)}\omega)u\right\|>
  Ce^{k(\lambda_{j}-\epsilon)}.
\end{eqnarray}
We remark that similar uniform lower bound estimates appear in the paper
of Barreira and Silva \cite{BS05}. Using Lemma \ref{lem:revexps}, Claim
\ref{claim:backwards} and Claim \ref{claim:Birkhoff} there exists a $C'$
depending on $\omega$ such that for all $k$,
\begin{eqnarray*}
  \max_{u\in S_1}\left\|A_{00}^{(k)}(\sigma^{-k}\omega)
    A_{10}(\sigma^{-(k+1)}\omega)u\right\|\le
  C'e^{k(\lambda_{j+1}+\epsilon)}e^{\epsilon k}.
\end{eqnarray*}
Combining the estimates we see
\begin{eqnarray*}
  g_n(\omega)\le\frac{C'}{C}\sum_{k=0}^{n-1}
  e^{k(\lambda_{j+1}-\lambda_j+3\epsilon)}.
\end{eqnarray*}
Since $3\epsilon<\lambda_j-\lambda_{j+1}$ it follows that defining
$M(\omega)=\sup_n g_n(\omega)$, one has $M(\omega)<\infty$ for almost all
$\omega$.

We define a distance $D$ between two subspaces of $\R^d$ of the same
dimension by the Hausdorff distance of their intersections with the
unit ball $B_1$ in $\R^d$. We now estimate
$D\left(\Wpush{j}{\omega}{n},\Wpush{j}{\omega}{m}\right)$ for $m>n$.

Let $x$ belong to the unit sphere of $\Wpush{j}{\omega}{n}$ (the distance
is always maximized by points on the boundary). Then
$x=A^{(n)}(\sigma^{-n}\omega)u$ for some $u\in \U{j}{\sigma^{-n}\omega}$.
Since for almost all $\omega$, the matrix
$A_{11}^{(m-n)}(\sigma^{-m}\omega)$ is invertible, there exists almost
surely a $u'\in \U{j}{\sigma^{-m}\omega}$ such that
$A^{(m-n)}(\sigma^{-m}\omega)u'=u+z$ where $z\in \V{j+1}{\sigma^{-n}\omega}$. 
Let $v'$ be the coordinates of $u'$ with respect to the basis $B_1(\omega)$. 
Then $\|A_{10}^{(m-n)}(\sigma^{-m}\omega)v'\|=\|z\|$ and 
$\|A_{11}^{(m-n)}(\sigma^{-m}\omega)v'\|=\|u\|$. It follows that 
$\|z\|\le M(\sigma^{-n}\omega)\|u\|$. Let
$y=A^{(m)}(\sigma^{-m}\omega)u'$ so that $y\in \Wpush{j}{\omega}{m}$. We
then have $y=x+A^{(n)}(\sigma^{-n}\omega)z$. By Claim
\ref{claim:backwards} we have
\begin{align}\label{eqn:UB}
  \|A^{(n)}(\sigma^{-n}\omega)z\|&\le Ce^{(\lambda_{j+1}+\epsilon)n}\|z\|  \nonumber\\
  &\le Ce^{(\lambda_{j+1}+\epsilon)n}M(\sigma^{-n}\omega)\|u\|
\end{align}
for a $C$ depending only on $\omega$. On the other hand,
\eqref{eq:A11lower} implies that
\begin{align}\label{eqn:LB}
  1=\|x\|=\|A^{(n)}(\sigma^{-n}\omega)v'\|\ge C'e^{(\lambda_{j}-\epsilon)n}\|u\|
\end{align}
for another $C'$ depending just on $\omega$. Let $K=C/C'$ and
$\alpha=\lambda_{j}-\lambda_{j+1}-2\epsilon>0$. 
Dividing (\ref{eqn:UB}) by (\ref{eqn:LB}) we see
\begin{eqnarray*}
  \|y-x\| =\|A^{(n)}(\sigma^{-n}\omega)z\|\le Ke^{-\alpha n} M(\sigma^{-n}\omega).
\end{eqnarray*}
The closest point of $\Wpush{j}{\omega}{m}\cap B_1$ to $x$ is just the
orthogonal projection of $x$ onto $\Wpush{j}{\omega}{m}$ (which lies in
$B_1$) so that the distance from $x$ to $\Wpush{j}{\omega}{m}\cap B_1$ is
bounded above by $\|y-x\|$ which in turn is bounded above by $Ke^{-\alpha
n} M(\sigma^{-n}\omega)$.

Conversely let $y\in B_1\cap \Wpush{j}{\omega}{m}$.
Then $y=A^{(m)}(\sigma^{-m}\omega)u'$ for some $u'\in
\U{j}{\sigma^{-m}\omega}$. Let $A^{(m-n)}(\sigma^{-m}\omega)u'$ be
decomposed into $u+z$ with $u\in \U{j}{\sigma^{-n}\omega}$ and $z\in
\V{j+1}{\sigma^{-n}\omega}$. Let $x=A^{(n)}(\sigma^{-n}\omega)u$. 
Since $\sup_n g_n(\omega)=M(\omega)$, we have 
$\|z\|\le M(\sigma^{-n}\omega)\|u\|$.
So $\|A^{(n)}(\sigma^{-n}\omega)z\|
\leq KM(\sigma^{-n}\omega)e^{-\alpha n}\|A^{(n)}(\sigma^{-n}\omega)u\|$. 
We also have 
\begin{align}\label{eqn:UBu}
\|A^{(n)}(\sigma^{-n}\omega)u\| &\le
\|A^{(n)}(\sigma^{-n}\omega)(u+z)\|+\|A^{(n)}(\sigma^{-n}\omega)z\|  \nonumber\\ 
&\le 1+KM(\sigma^{-n}\omega)e^{-\alpha n}\|A^{(n)}(\sigma^{-n}\omega)u\|.
\end{align}

So \mbox{$\|A^{(n)}(\sigma^{-n}\omega) u\|\le 
1/(1-KM(\sigma^{-n}\omega)e^{-\alpha n})$},
provided $KM(\sigma^{-n}\omega)e^{-\alpha n}<1$.
Combining this estimate with (\ref{eqn:UBu}) gives
\begin{align*}
\|x-y\|=\|A^{(n)}(\sigma^{-n}\omega)z\|\le
\frac{KM(\sigma^{-n}\omega)e^{-\alpha n}}
{1-KM(\sigma^{-n}\omega)e^{-\alpha n}}.
\end{align*} 

As before it follows that the closest point of $\Wpush{j}{\omega}{n}\cap
B_1$ to $y$ is at a distance at most $KM(\sigma^{-n}\omega)e^{-\alpha
n}/(1-KM(\sigma^{-n}\omega)e^{-\alpha n})$.  In particular, provided that
\mbox{$KM(\sigma^{-n}\omega)e^{-\alpha n}<1$}, we have
\begin{eqnarray*}
  D\left(\Wpush{j}{\omega}{n},\Wpush{j}{\omega}{m}\right)\le
  \frac{KM(\sigma^{-n}\omega)e^{-\alpha n}}
  {1-KM(\sigma^{-n}\omega)e^{-\alpha n}}.
\end{eqnarray*}
Obviously for $m,m'>n$ one then has
\begin{eqnarray*}
  D\left(\Wpush{j}{\omega}{m},\Wpush{j}{\omega}{m'}\right)\le
  \frac{2KM(\sigma^{-n}\omega)e^{-\alpha n}}{1-KM(\sigma^{-n}\omega)e^{-\alpha n}}.
\end{eqnarray*}
Since $M(\omega)$ is measurable and $\sigma$ is ergodic, there exist
for almost all $\omega$ arbitrarily large values of $n$ such that
$M(\sigma^{-n}\omega)<A$ for some fixed $A>0$. It follows that the
sequence of subspaces is Cauchy and hence convergent to a subspace
$\W{j}{\omega}$.

Let $x$ belong to the unit sphere of $W_j^{(n)}(\omega)$. Then 
$x=A^{(n)}(\sigma^{-n}\omega)u$. As before, writing $x$ as $y+z$ 
with $y\in U_j(\omega)$ and $z\in V_{j+1}(\omega)$, we have 
$\|z\|\le M(\omega)\|y\|$. Since $\|y\|^2+\|z\|^2=1$, we have 
$\|y\|^2(1+M(\omega)^2)\ge 1$ so that $\|y\|\ge 1/\sqrt{1+M(\omega)^2}=B$. 
Thus each point of the unit
sphere of $\W{j}{\omega}$ has a component in $\U{j}{\omega}$ of norm at
least $B$.  It follows that $\V{j}{\omega}=\V{j+1}{\omega}\oplus
\W{j}{\omega}$.

This completes the proof.
\end{proof}

\begin{acknowledgements}
The authors wish to thank Christopher Bose for very helpful 
discussions concerning Section 8.
GF and SL acknowledge support by the Australian Research Council 
Discovery Project DP0770289. AQ acknowledges partial support from 
the Natural Sciences and Engineering Research Council of Canada, 
and support while visiting the University of New South Wales from 
the Australian Mathematical Sciences Institute and from the 
Australian Research Council Centre of Excellence for Mathematics 
and Statistics of Complex Systems.
\end{acknowledgements}

\end{document}